\documentclass{amsart}
\usepackage{amsmath}
\usepackage{amssymb}
\usepackage[all]{xy}

\theoremstyle{plain}
\newtheorem{thm}{Theorem}[section]
\newtheorem{thm*}{Theorem}[section]

\newtheorem{prop}[thm]{Proposition}
\newtheorem{lemma}[thm]{Lemma}

\theoremstyle{definition}
\newtheorem{defn}[thm]{Definition}
\newtheorem{remark}[thm]{Remark}
\newtheorem{ex}[thm]{Example}

\numberwithin{equation}{thm}

\newcommand{\cN}{\mathcal N}

\def\JType{\operatorname{JType}\nolimits}
\def\Spec{\operatorname{Spec}\nolimits}
\def\Ker{\operatorname{Ker}\nolimits}

\def\Rad{\operatorname{Rad}\nolimits}

\def\rk{\operatorname{Rk}\nolimits}

\def\Proj{\operatorname{Proj}\nolimits}
\def\dim{\operatorname{dim}\nolimits}

\def\stmod{\operatorname{stmod}\nolimits}

\newcommand{\cG}{\mathcal G}
\newcommand{\bG}{\mathbb G}

\newcommand{\bZ}{\mathbb Z}
\newcommand{\bF}{\mathbb F}

\newcommand{\gl} {\mathfrak {gl}}

\newcommand{\ul}{\underline}

\def\rk{\operatorname{rk}\nolimits}
\def\pr{\operatorname{pr}\nolimits}

\def\Spec{\operatorname{Spec}\nolimits}
\def\sl2{\operatorname{SL_{2(2)}}\nolimits}
\def\Ga2{\operatorname{\mathbb G_{\rm a(2)}}\nolimits}
\def\SL{\operatorname{SL}\nolimits}
\def\GL{\operatorname{GL}\nolimits}
\def\triv{\operatorname{triv}\nolimits}
\def\bH{\operatorname{H^\bu}\nolimits}
\def\HHH{\operatorname{H}\nolimits}
\def\Ext{\operatorname{Ext}\nolimits}
\def\Ann{\operatorname{Ann}\nolimits}
\def\End{\operatorname{End}\nolimits}
\def\Hom{\operatorname{Hom}\nolimits}

\def\proj{\operatorname{proj}\nolimits}

\newcommand{\Gar}{\mathbb G_{a(r)}}

\newcommand{\bN}{\mathbb N}

\newcommand{\bu}{\bullet}

\date{April 26, 2010}

\begin{document}

 \title{Generalized support varieties for finite group schemes}

\author[Eric M. Friedlander and Julia Pevtsova] 
{Eric M. Friedlander$^*$ and 
Julia Pevtsova$^{**}$}
\dedicatory{to Andrei Suslin, with great admiration}

\address {Department of Mathematics, University of Southern California,
Los Angeles, CA }
\email{ericmf@usc.edu}

\address {Department of Mathematics, University of Washington, 
Seattle, WA}
\email{julia@math.washington.edu}

\thanks{$^*$ partially supported by the NSF  \# DMS 0909314}
\thanks{ $^{**}$ partially supported by the NSF  \# DMS 0800950}



\begin{abstract} 

We construct two families of refinements of the (projectivized) support variety of a finite dimensional 
module $M$ for a finite group scheme $G$.  For an arbitrary finite group scheme, we associate a family 
of {\it non maximal rank varieties} $\Gamma^j(G)_M$, $1\leq j \leq p-1$, to a $kG$-module $M$.  
For $G$ infinitesimal, we construct a finer family of locally closed subvarieties $V^{\ul a}(G)_M$ of 
the variety of one parameter subgroups of $G$ for any  partition $\ul a$  of $\dim M$.   For an 
arbitrary finite group scheme $G$, a $kG$-module $M$ of constant rank, and a cohomology class $\zeta$ in 
$\HHH^1(G,M)$ we introduce the {\it zero locus} $Z(\zeta) \subset \Pi(G)$. 
We show that $Z(\zeta)$ is a closed subvariety, and relate it to 
the non-maximal rank varieties.  
We also extend the construction of $Z(\zeta)$ to an arbitrary extension class $\zeta \in \Ext^n_G(M,N)$ 
whenever $M$ and $N$ are $kG$-modules  of constant Jordan type. 


\end{abstract}

\maketitle



\section{Introduction}

In the remarkable papers \cite{Q}, D. Quillen identified the spectrum of the (even dimensional)
cohomology of a finite group $\Spec \HHH^\bu(G,k)$ where $k$ is some field of characteristic $p$
dividing the order of the group.  The variety $\Spec \HHH^\bu(G,k)$ is the ``control space" for 
certain geometric invariants of finite dimensional $kG$-modules. These invariants,  
{\it cohomological support varieties} and {\it rank varieties}, were initially introduced and studied in \cite{AE} and \cite{C}.
Over the last twenty five  years, many authors have been investigating these varieties inside  $\Spec \HHH^\bu(G,k)$ in order to provide insights into the structure, behavior, and properties of $kG$-modules.  The initial theory for finite groups has
been extended to a much more general family of finite group schemes, starting with the work of \cite{FPa} for
$p$-restricted Lie algebras. The resulting theory of support varieties for modules for finite group schemes satisfies all
of the axioms of a ``support data" of tensor triangulated categories as defined in \cite{Bal}.
Thus, for example, this theory  provides
a classification of tensor--ideal, thick subcategories of the stable module category of a finite group scheme $G$.

In this present paper, we embark on a different perspective of geometric invariants for $kG$-modules for a finite group scheme $G$.
We introduce a new family of invariants, ``generalized support varieties", which stratify the support variety of a finite dimensional $kG$-module $M$.
As finer invariants, they capture more structure of a module $M$ and can 
distinguish between modules with the same support varieties.  In particular, 
the generalized support varieties are always proper subvarieties of the control space 
$\Spec \HHH^\bu(G,k)$ whereas the support variety often coincides with the entire control space.  
On the other hand, they necessarily lack certain good 
behavior with respect to tensor products and distinguished triangles in the stable module category of $kG$.  
However, as we shall try to convince the reader, these varieties provide interesting and useful 
tools in the further study of
the representation theory of finite groups and their generalizations.

Since the module category of a finite group scheme $G$ is wild except 
for very special $G$, our goals are necessarily
more modest than the classification of all (finite dimensional) $kG$-modules.  
Two general themes that we follow when introducing our new varieties associated to 
representations are the formulation of invariants
which distinguish various known classes of modules and 
the construction of modules with specified invariants.

In Section \ref{recollect},   
we summarize some of our earlier work, and that of others, concerning
support varieties of $kG$-modules.  We emphasize the formulation of support
varieties in terms of $\pi$-points, since the fundamental structure underlying our
new invariants is the scheme $\Pi(G)$ of equivalence classes of $\pi$-points.  Also
in this section, we recall maximal Jordan types of $kG$-modules and the
non-maximal subvariety $\Gamma(G)_M \subset M$ refining the support
variety $\Pi(G)_M$ for a finite dimensional $kG$-module $M$.

If $G$ is an infinitesimal group scheme, one formulation of support varieties is 
in terms of the affine scheme $V(G)$ of infinitesimal subgroups of $G$.  
For any Jordan type $\ul a§ = \sum_{i=1}^p a_i [i]$ and any finite dimensional 
$kG$-module $M$ (with $G$ infinitesimal), we associate in Section \ref{refined} subvarieties 
$V^{\leq \ul a}(G)_M$ and $V^{\ul a}(G)_M$ of $V(G)$.   Determination of these 
refined support varieties is enabled by earlier computations of the global $p$-nilpotent 
operator $\Theta_G: M \otimes k[V(G)] \to M \otimes k[V(G)]$ which was introduced and 
studied in \cite{FP3}.

We require a refinement of one of the main theorems of \cite{FPS} recalled as 
Theorem \ref{maximal}.  Section \ref{j-type} outlines the original proof due to A. Suslin
and the authors, and points out the minor modifications required to establish the fact
that whether or not a $kG$-module has maximal $j$-type at a $\pi$-point depends
only upon the equivalence class of that $\pi$-point (Theorem \ref{gen}). This is the key result 
needed to establish that the generalized support varieties are well--defined for all finite group schemes. 

In Section \ref{arbitrary}, we consider closed subvarieties $\Gamma^j(G)_M \subset \Pi(G)$
for any finite group scheme, finite dimensional $kG$-module $M$, and integer $j, 1 \leq j < p$, {\it the non maximal rank varieties}.
We establish some properties  of these varieties and work out a few examples to 
suggest how these invariants
can distinguish certain non-isomorphic $kG$-modules.

In the concluding Section \ref{ext-class}, we employ $\pi$-points to associate a closed subvariety
$Z(\zeta) \subset \Pi(G)$ to a cohomology class $\zeta \in \HHH^1(G,M)$ provided 
that $M$ is a $kG$-module of constant rank. One of the key properties of $Z(\zeta)$ is that  
$Z(\zeta) = \emptyset$ if and only if the extension $0 \to M \to E_\zeta \to k \to 0$
satisfies the condition that $E_\zeta$ is also a $kG$-module of constant rank. We show that
 $Z(\zeta)$ is often homeomorphic to $\Gamma^1(G)_{E_\zeta}$ which allows us to conclude  
that $Z(\zeta)$ is closed. Taking $M$ to be an odd degree Heller shift of the trivial module $k$, we recover the
familiar zero locus of a class in $\HHH^{2n}(G,k)$ in the special case $M = k$.   
 Finally, we generalize
this construction to extension classes $\xi \in \Ext^n_G(M,N)$ for $kG$-modules $M$ and $N$ 
of constant Jordan type and any $n \geq 0$.

	We abuse terminology in this paper by referring to a (Zariski) closed subset of an
affine or projective variety as a subvariety.  Should one wish, one could always impose
the reduced scheme structure on such ``subvarieties".

We would  like to thank Jon Carlson for pointing out to us that maximal ranks do not 
behave well under tensor product, Rolf Farnsteiner for his insights into components of the  Auslander-Reiten quiver, 
and the referee for several useful comments. 
The second author gratefully acknowledges the support of MSRI during her postdoctoral appointment there.


\section{Recollection of $\Pi$-point schemes and support varieties}
\label{recollect}

Throughout, $k$ will denote an arbitrary field of characteristic $p > 0$.  Unless
explicit mention is made to the contrary, $G$ will denote a finite group scheme over $k$
with finite dimensional coordinate algebra $k[G]$.   We denote by $kG$ the 
Hopf algebra dual to $k[G]$, and refer to $kG$ as the group algebra of $G$.
Thus, (left) $kG$-modules are naturally equivalent to (left) $k[G]$-comodules,  which are equivalent to (left) 
rational $G$-modules (see \cite[ch.1]{Jan}).  If $M$ is a $kG$-module
and $K/k$ is a field extension, then we denote by $M_K$ the $KG$-module 
obtained by base change.

We shall identify 
$\HHH^*(G,k)$ with $ \HHH^*(kG,k)$.

\begin{defn} (\cite{FP2}) 
The {\it $\Pi$-point scheme} of a finite group scheme $G$ is the $k$-scheme of finite type
whose points are equivalence classes of $\pi$-points of $G$ and whose scheme structure
is  defined in terms of the category of $kG$-modules.

In more detail,
\begin{enumerate}
\item
A $\pi$-point of $G$ is a (left)
flat map of $K$-algebras $\alpha_K: K[t]/t^p \to KG$ for some field extension 
$K/k$ with the property that there exists a unipotent abelian subgroup scheme
$i: C_K \subset G_K$ defined over $K$ such that $\alpha_K$ factors through 
$i_*: KC_K \to KG_K = KG$.
\item If $\alpha_K: K[t]/t^p \to KG$, $\beta_L: L[t]/t^p \to LG$ are two $\pi$-points of $G$, then $\alpha_K$ 
is said to be a {\it specialization} of $\beta_L$ ,  provided that 
for any finite dimensional $kG$-module $M$, $\alpha_K^*(M_K)$ being free  as $K[t]/t^p$-module
implies that $\beta^*_L(M_L)$ is free as $L[t]/t^p$-module.
\item
Two $\pi$-points $\alpha_K: K[t]/t^p \to KG, ~ \beta_L: L[t]/t^p \to LG$ 
are said to be {\it equivalent},  written $\alpha_K \sim \beta_L$, if they satisfy 
the following condition for all finite dimensional $kG$-modules $M$:
$\alpha_K^*(M_K)$ is free  as $K[t]/t^p$-module if and only if 
$\beta^*_L(M_L)$ is free as $L[t]/t^p$-module.
\item
A subset $V \subset \Pi(G)$ is closed if and only if there exists a finite dimensional 
$kG$-module $M$ such that $V$ equals
$$\Pi(G)_M = \{[\alpha_K] \, |  \, \alpha^*_K(M_K) \text{ is not free as } 
K[t]/t^p-\text{module} \}$$

\noindent
The closed subset $\Pi(G)_M  \subset \Pi(G)$ is called the {\it $\Pi$-support of $M$}.

\item The topological space $\Pi(G)$ of equivalence classes of $\pi$-points can be endowed with a scheme structure based on representation theoretic properties of $G$ (see \cite[\S7]{FP2}).
\end{enumerate}
\end{defn}

\noindent
We denote by 
$$
\bH(G,k) = 
\begin{cases}
\HHH^*(G,k), & \text{if $p = 2$,}\\
\HHH^{\rm ev}(G,k) & \text{if $p > 2$}. \\
\end{cases}
$$ 
The {\it  cohomological support variety} $|G|_M$ of a  $kG$-module $M$ 
is the closed subspace of $\Spec \HHH^\bu(G,k)$ defined as the variety of 
the ideal $\Ann_{\HHH^\bu(G,k)}\Ext_{G}^*(M,M) \subset \HHH^\bu(G,k)$. 

\begin{thm} \cite[7.5]{FP2}
\label{cohom}
Let $G$ be a finite group scheme, and $M$ be a finite dimensional $kG$-module.  
Denote by $\Proj \HHH^\bu(G,k)$ the projective $k$-scheme associated to the commutative,
graded $k$-algebra $ \HHH^\bu(G,k)$.  Then there is an isomorphism of $k$-schemes
$$\Phi_G: \Proj \HHH^\bu(G,k) \simeq \Pi(G)$$ 
which restricts to a homeomorphism  of closed subspaces
$$\Proj (|G|_M) \simeq \Pi(G)_M$$
for all finite dimensional $kG$-modules $M$.
\end{thm}

We (implicitly) identify $\Proj \HHH^\bu(G,k)$ with $\Pi(G)$ via this isomorphism.

We consider the {\it stable module category} $\stmod kG$.  Recall that the Heller shift $\Omega(M)$ 
of $M$ is the kernel of the minimal projective cover $P(M) \twoheadrightarrow M$, and the inverse 
Heller shift $\Omega^{-1}(M)$  is the cokernel of the embedding of $M$ into its injective hull, 
$M \hookrightarrow I(M)$.

The objects of $\stmod kG$  are finite dimensional $kG$-modules.  The morphisms are equivalence 
classes where two morphisms are equivalent if they differ by a morphism which factors through a 
projective module, 
$$\Hom_{\stmod kG}(M,N) = \Hom_{kG}(M,N) / {\rm PHom}_{kG} (M,N).$$
The stable module category has a tensor triangulated structure: the triangles are induced by exact 
sequences, the shift operator is given by the inverse Heller operator $\Omega^{-1}$, and the tensor 
product is the standard tensor product in the category of $kG$-modules.  Two $kG$-modules 
$M$, $N$ are stably isomorphic if and only if they are isomorphic as $kG$-modules up to a 
projective direct summand.

The association $M \mapsto \Pi(G)_M$
fits the abstractly defined ``theory of supports" for the stable module category of $G$ 
(as defined in \cite{Bal}).
Some of the basic properties of this theory are summarized in the next theorem (see \cite{FP2}).

\begin{thm}
\label{properties}
Let $G$ be a finite group scheme and let $M, N$ be finite dimensional $kG$-modules.
\begin{enumerate}
\item $\Pi(G)_M = \emptyset$ if and only if $M$ is projective as a $kG$-module.
\item  $\Pi(G)_{M\oplus N} = \Pi(G)_M \cup \Pi(G)_N$.
\item $\Pi(G)_{M\otimes N} = \Pi(G)_M \cap \Pi(G)_N$.
\item $\Pi(G)_M = \Pi(G)_{\Omega M}$.
\item If $M \to N \to Q \to \Omega^{-1}M$ is an exact triangle in the stable module category 
$\stmod(kG)$  
then $\Pi(G)_N \subset \Pi(G)_M \cup \Pi(G)_Q$.
\item  If $p$ does not divide the dimension of $M$, then $\Pi(G)_M = \Pi(G)$.
\end{enumerate}
\end{thm}

The last property of Theorem \ref{properties} indicates that $M \mapsto \Pi(G)_M$ is
a somewhat crude invariant.  

We next recall the use of Jordan types in order to refine this theory.
The isomorphism type of a finite dimensional $k[t]/t^p$-module $M$ is  said
to be the Jordan type of $M$.  We denote the  Jordan type of $M$ by $\JType(M)$,
and write $\JType(M) =  \sum_{i=1}^p a_i[i]$; in other words, as a $k[t]/t^p$-module 
$M \ \simeq  \bigoplus_{i=1}^p ([i])^{\oplus a_i}$ where $[i] = k[t]/t^i$.   Thus, we may (and
will) view a Jordan type $\JType(M)$ as a partition of $m= \dim M$ into 
subsets each of which has cardinality $\leq p$.

We shall compare Jordan types using the {\it dominance order}. Let 
$\underline n = [n_1 \leq n_2 \leq \ldots \leq n_k]$, $\underline m = [m_1 \leq m_2 \leq \ldots \leq m_k]$ 
be two partitions  of $N$. Then $\underline n$ {\it dominates} $\underline  m$, written $\underline n \geq \underline  m$,
iff

\begin{equation}
\label{ineq}
\sum_{i = j}^k n_i \  \geq  \ \sum_{i = j}^k m_i .
\end{equation} 
for all $j, \, 1 \leq j \leq k$.  
For $k[t]/t^p$-modules $M, N$, we say that $\JType(M) \geq \JType(N)$ if the partition corresponding to $\JType(M)$
dominates the partition  corresponding to $\JType(N)$.  The dominance order on Jordan types can be reformulated in the following way.

\begin{lemma}
\label{reform}  Let $M$, $N$ be $k[t]/t^p$-modules of dimension $m$. 
Then $\JType(M) \geq \JType(N)$ if and only if
$$\rk (t^j, M) \geq \rk (t^j, N)$$
for all $j, \, 1 \leq j < p$, where $\rk (t^j, M)$ denotes the rank of the operator $t^j$ on $M$. 
\end{lemma}

\begin{proof} If $\JType(M) = \sum\limits_{i=1}^p a_i[i]$, then
\begin{equation}\label{rank} 
\rk (t^j, M) = \sum_{i = j+1}^p a_i (i-j).
\end{equation}
The statement now follows from \cite[6.2.2]{CM}.
\end{proof}

The following theorem plays a key role in our formulation of geometric invariants for a
$kG$-module $M$ that are finer than the $\Pi$-support $\Pi(G)_M$.   In Section \ref{j-type}, we 
outline the proof of this theorem in order to prove the related, but sharper, Theorem \ref{gen}. 
We say that a $\pi$-point $\alpha_K$ has maximal Jordan type on a $kG$-module $M$ 
if there does not exist a $\pi$-point $\beta_L$ such that $\JType(\alpha_K^*(M_K)) < \JType(\beta_L^*(M_L))$.   

\begin{thm}\cite[4.10]{FPS}
\label{maximal}
Let $G$ be a finite group scheme over $k$ and $M$ a finite 
dimensional $kG$-module.  Let $\alpha_K: K[t]/t^p \to KG$ be a $\pi$-point of $G$ which has  
maximal Jordan type on $M$. Then for any $\pi$-point 
$\beta_L: L[t]/t^p \to LG$ which specializes 
to $\alpha_K$,
the Jordan type of $\alpha_K^*(M_K)$ equals the Jordan type of  
$\beta_L^*(M_L)$; in particular, if $\alpha_K \sim \beta_L$, then
the Jordan type of $\alpha_K^*(M_K)$ equals the Jordan type of $\beta_L^*(M_L)$.
\end{thm}

The following class of $kG$-modules was introduced in \cite{CFP} and further studied in \cite{CF},
\cite{CFS}, \cite{B1}, \cite{B2}.  

\begin{defn}
\label{constant}
A finite dimensional $kG$-module $M$ is said to be of {\it constant Jordan type} if
the Jordan type of $\alpha_K^*(M_K)$ is the same for every $\pi$-point $\alpha_K$ of
$G$.   By Theorem \ref{maximal}, $M$ has constant Jordan type $\ul a$ if and only if
for each point of $\Pi(G)$ there is some representative $\alpha_K$ of that point with
$\JType(M) = \ul a$.
\end{defn}
Theorem  \ref{maximal} justifies the following definition (see \cite[5.1]{FPS}).

\begin{defn} (\cite[5.1]{FPS})
\label{nmax}  Let $M$ be a finite dimensional representation of a finite group scheme $G$. 
We define $\Gamma(G)_M\subset \Pi(G)$  to be the subset of equivalence classes of 
$\pi$-points $\alpha_K: K[t]/t^p \to KG$  such that $\JType(\alpha_K^*(M_K))$ is not maximal 
among Jordan types $\JType(\beta^*_L(M_L))$ where $\beta_L$ runs over all $\pi$-points of $G$. 
\end{defn}

To conclude this summary, we recall certain properties of the association
$M \ \mapsto \ \Gamma(G)_M$.

\begin{prop}
\label{gamma}
Let $G$ be a finite group scheme and let $M, N$ be finite dimensional $kG$-modules.
Then $\Gamma(G)_M \subset \Pi(G)$ is a closed subvariety satisfying the following properties:
\begin{enumerate}
\item If $M$ and $N$ are stably isomorphic, then $\Gamma(G)_M = \Gamma(G)_N$.
\item $\Gamma(G)_M \subset \Pi(G)_M$ with equality if and only if $\Pi(G)_M \not= \Pi(G)$.
\item $\Gamma(G)_M$ is empty if and only if $M$ has constant Jordan type.
\item If $M$ has constant Jordan type, then $\Gamma(G)_{M\oplus N} = \Gamma(G)_N$. 
\item If $\Pi(G)$ is irreducible, then $N$ has constant non-projective Jordan type if and 
only if $\Gamma(G)_{M\otimes N} = \Gamma(G)_M$  for any $kG$-module $M$. 
\item If $\Pi(G)$ is irreducible, then 
$$\Gamma(G)_{M\otimes N} \ = \ (\Gamma(G)_M \cup \Gamma(G)_N) \cap
 (\Pi(G)_M \cap \Pi(G)_N)
.$$
\end{enumerate}
\end{prop}

\begin{proof} 
If $M$ and $N$ are stably isomorphic then $M = N \oplus P$ (or vice versa) with $P$ projective. 
Since  projective modules have constant Jordan type, (1) becomes a  special case of (4).  
The fact that $\Gamma(G)_M \subset \Pi(G)$ is closed is proved in \cite[5.2]{FPS}.  
Properties (2) and (3) follow essentially from definitions.   Property (4) follows from the 
additivity of the dominance order.
Properties (5)  and (6) are the statements of \cite[4.9]{CFP} and \cite[4.7]{CFP} respectively.  
\end{proof}


\section{Refined support varieties for infinitesimal group schemes}
\label{refined}

Before considering refinements of $\Gamma(G)_M \subset \Pi(G)$
in Section \ref{j-type} for a general finite group scheme $G$, we specialize 
in this section to infinitesimal group
schemes and work with the affine variety $V(G)$.  First, we recall some definitions and 
several fundamental results from \cite{SFB1}, \cite{SFB2}.

A finite group scheme is called {\it infinitesimal} if its coordinate algebra $k[G]$ is local. 
Important examples of infinitesimal group schemes are Frobenius kernels of algebraic groups 
(see \cite{Jan}).  An infinitesimal group scheme is said to have height less or equal to $r$ if for 
any $x$ in $\Rad(k[G])$, $x^{p^r} =0$.  
 A {\it one-parameter subgroup} of height $r$  of $G$ over a commutative $k$-algebra 
 $A$ is a map of group schemes over $A$ of the form $\mu: \bG_{a(r),A} \to G_A$. 
 Here, $\bG_{a(r),A}, \ G_A$ are group schemes over $A$ defined as the base changes
 from $k$ to $A$ of $\Gar, \ G$.

Let $\bG_a$ be the additive group, and $\Gar$ be the $r$-th Frobenius kernel of $\bG_a$.  
Then $k[\Gar] = k[T]/T^{p^r}$, and $k\bG_{a(r)} = k[u_0,\ldots,u_{r-1}]/(u_0^p,\ldots,u_{r-1}^p)$, 
indexed so that the 
Frobenius map $F: \Gar \to \Gar$ satisfies $F_*(u_i) = u_{i-1}, i >0;
F_*(u_0) = 0$.  We define
\begin{equation}
\label{epsilon}
\epsilon: k[u]/u^p \ \to \ k\bG_{a(r)} = k[u_0,\ldots,u_{r-1}]/(u_0^p,\ldots,u_{r-1}^p)
\end{equation}
to be the map sending $u$ to $u_{r-1} \in k\bG_{a(r)}$.  Thus,
$\epsilon$ is a map of group algebras but not of Hopf algebras in general.  
In fact,  the map $\epsilon$ is induced by a group scheme homomorphism 
if and only if $r=1$ in which case $\epsilon$ is an isomorphism.

\begin{thm} \cite{SFB1}
\label{1-param}
Let $G$ be an infinitesimal group scheme of height $\leq r$.  
Then there is an affine group scheme
$V(G)$ which represents the functor sending a commutative $k$-algebra $A$ to the 
set $\Hom_{{\rm gr. sch}/A}(\bG_{a(r),A}, G_A)$.  
\end{thm}
\noindent
Thus, a point $v \in V(G)$ naturally 
corresponds to a 1-parameter subgroup $$\xymatrix{\mu_v: \bG_{a(r),k(v)} \ar[r]& G_{k(v)}}$$
where $k(v)$ is the residue field of $v$.

\begin{thm} \cite{SFB2}
(1). The closed subspaces of $V(G)$ are the subsets of the form
$$V(G)_M = \{ v \in V(G)\, |  \, \epsilon^*\mu_v^*(M_{k(s)}) \text{ is not free as a module over } 
k(v)[u]/u^p \}$$for some finite dimensional $kG$-module $M$.\\[1pt]
(2). There is a natural
$p$-isogeny 
$V(G) \longrightarrow \Spec \HHH^\bu(G,k)$
which restricts to a homeomorphism
$V(G)_M \simeq |G|_M$
for any finite dimensional $kG$-module $M$. 
\end{thm}
 Theorem~\ref{cohom} implies that the spaces $\Pi(G)$ and $\Proj k[V(G)]$ are also homeomorphic 
(see \cite{FP2} for a natural direct relationship between $\Pi(G)$ and $V(G)$ for an infinitesimal group 
scheme).

Let $\mu_{v*}:k(v)\Gar \to k(v)G$ be the map on group algebras  induced  by the one-parameter 
subgroup $\mu_v: \Gar \to G$. We denote by $\theta_v$ the nilpotent element of $k(v)G$ which 
is the image $u$ under the composition
$$ \xymatrix{k(v)[u]/u^p \ar[r]^-\epsilon & k(v)[u_0,\ldots,u_{r-1}]/(u_0^p,\ldots,u_{r-1}^p) \ar[r]^-{\mu_{v*}} 
& k(v)G}.$$  
So,  $\theta_v = \mu_{v*}(u_{r-1}) \in k(v)G$. 
For a given $kG$-module $M$  we also let 
$$\theta_v: M_{k(v)} \to M_{k(v)}
$$
denote the associated $p$-nilpotent endomorphism.  Thus, 
$\JType(\epsilon^*\mu_{v}^*(M_{k(v)}))$ is the Jordan type of $\theta_v$ on
$M_{k(v)}$.

\begin{defn}
Let $M$ be a $kG$-module of dimension $m$.  We define the {\it local Jordan type
function}
\begin{equation}
\label{JType}
\JType_M: V(G) \ \to \ \bN^{\times p}, 
\end{equation}
by sending $v$ to  
$(a_1,\ldots,a_p)$, where
 $(\theta_v)^*(M_{k(v)}) \simeq \sum_{i=1}^p a_i[i]$.
\end{defn}

\begin{defn}
For a given $\ul a = (a_1,\ldots,a_p) \in \bN^{\times p}$, we define 
$$V^{\ul a}(G)_M = \{ v \in V(G) \, | \, \JType_M(v) = \ul a \}, $$
$$V^{\leq \ul a}(G)_M = \{ v \in V(G) \, | \, \JType_M(v) \leq \ul a \}. $$
\end{defn}

\noindent As we see in the following example, $V^{\ul a}(G)_M $ is a generalization 
of a nilpotent orbit of the adjoint representation (and
$V^{\leq \ul a}(G)_M $ is a generalization of an orbit closure).

\begin{ex}
\label{elemex}
Let $G = \GL_{N(1)}$ and let $M$ be the standard $N$-dimensional
 representation of $\GL_N$.  Then $\JType_M$ sends a $p$-nilpotent
matrix $X$ to its Jordan Jordan type as an endomorphism of $M$.  Consequently, 
$\JType_M$ has image inside $\bN^{\times p}$ consisting of those $p$-tuples
$\ul a = (a_1,\ldots,a_p)$ such that $\sum_i a_i \cdot i = N$.   The locally closed
subvarieties $V^{\ul a}(G)_M
\subset \cN_p(\gl_N)$ are precisely the adjoint $\GL_N$-orbits inside the $p$-nilpotent
cone $\cN_p(\gl_N)$  of the Lie algebra $\gl_N$.
\end{ex}

\begin{ex}
Let $\zeta \in \HHH^{2i+1}(G,k)$ be a non-zero cohomology class of odd degree.
Let $L_\zeta$ be the Carlson module
defined as the kernel of the map $\Omega^{2i+1}(k) \to k$ corresponding to 
$\zeta$ (see \cite[II.5.9]{Ben}).  The module $\Omega^{2i+1}(k)$ has constant Jordan type $m[p] + [p-1]$.
Let $\ul a = m[p] + [p-2]$ and $\ul b = (m-1)[p] + 2[p-1]$.   Then the image of 
$\JType_{L_\zeta}$ equals $\{ \ul a, \ul b \} \subset \bN^{\times p}$.  Moreover, $V^{\ul a}(G)_{L_\zeta}$
is open in $V(G)$, with complement $V^{\ul b}(G)_{L_\zeta}$.
\end{ex}

\begin{remark}
\label{re:theta}
An explicit determination of the global $p$-nilpotent operator $\Theta_M: M \otimes k[V(G)]
\to M\otimes k[V(G)]$ of \cite[2.4]{FP3} immediately determines 
the local Jordan type function $\JType_M$. Namely, to any $v \in V(G)$ we  associate a 
nilpotent linear operator $\theta_v: M_{k(v)} \to M_{k(v)}$ defined by $\theta_v = \Theta_M \otimes_{k(v)[V(G)]}{k(v)}$. 
The local Jordan type of $M$ at the point $v$ is precisely the Jordan type of the linear operator $\theta_v$. 

The reader should consult \cite{FP3} for many explicit examples of $kG$-modules $M$
for each of the four families of examples of infinitesimal group schemes: (i.) $G$ of height 1,
so that $M$ is a $p$-restricted module for $Lie(G)$; (ii.) $G = \Gar$; (iii.) $\GL_{n(r)}$;
and (iv.) $\SL_{2(2)}$.
\end{remark}

We provide a few elementary properties of these refined support varieties.

\begin{prop}
Let $M$ be a $kG$-module of dimension $m$ and let
 $\ul a = (a_1, \ldots, a_p)$ such that $\sum_{i=1}^p a_i\cdot i = m$.
\begin{enumerate}
\item  If $m = p\cdot m^\prime$, then
$V(G) \backslash\ V(G)_M \ = \ V^{(0,\ldots,0,m^\prime)}(G)_M$;
otherwise,  $V(G) = V(G)_M$.
\item $M$ has constant Jordan type if and only if $V(G)_M = V^{\ul a}(G)_M$
for some $\ul a \in \bN^{\times p}$ (in which case $\ul a$ is the Jordan type of $M$).
\item 
$V^{\leq \ul a}(G)_M \ = \ \{ v\in V(G) \ | \  \JType_M(v) \leq \ul a\}$
 is a closed subvariety of $V(G)$.
\item
$V^{\ul a}(G)_M$ is a locally
closed subvariety of $V(G)$, open in $V^{\leq \ul a}(G)_M$. 
 \item
 $V^{\leq \ul b}(G)_M \ \subseteq \ V^{\leq \ul a}(G)_M , \quad {\text if \ } 
 \ul b \leq \ul a$, where $``\leq"$ is the dominance order on Jordan types. 
 \end{enumerate}
\end{prop}

\begin{proof}  Properties (1) and (2) follow immediately from the definitions of  $V(G)_M$
in Theorem \ref{1-param} and of constant Jordan type in Definition \ref{constant}.  Property
(5) is immediate.

To prove (3), we utilize $\theta_v=\Theta_M \otimes_{k(v)[V(G)]}{k(v)}: M_{k(v)} \to M_{k(v)}$  
described in Remark~\ref{re:theta}. Applying Nakayama's Lemma as in \cite[4.11]{FP3} to 
$\Ker \{\Theta_M^j\}, \ 1 \leq j < p$, we conclude that $\rk( \theta_v^j, M), 1 \leq j \leq p-1$, 
is lower semi-continuous.  Consequently,
(\ref{ineq}) and Lemma~\ref{reform} imply that $V^{\leq \ul a}(G)_M$ is closed.

Property (4) follows from the observation that $V^{\ul a}(G)_M$ is the complement
inside $V^{\leq \ul a}(G)_M$ of the finite union $V^{< \ul a}(G)_M = 
\cup_{\ul a^\prime < \ul a} V^{\leq \ul a^\prime}$, which is closed by (3).
\end{proof}

It is often convenient to consider the {\it stable Jordan type} of a $k[t]/t^p$-module $M$:
if $a_1[1] + \ldots + a_p[p]$ is the Jordan type of $M$, then the stable Jordan type
of $M$ is $a_1[1] + \ldots +a_{p-1}[p-1]$ (equivalently, the isomorphism class of $M$ in the stable module category $\stmod k[u]/u^p$). We define the stable
local Jordan type function
$$\ul{\JType}_M: V(G) \to \bN^{\times p-1}, \quad v \mapsto (a_1, \ldots, a_{p-1})$$
by sending $v$ to the stable Jordan type of $\theta_v^*(M_{k(v)})$.

\vspace{0.1in}

The following proposition relates the Jordan type function for a module $M$ and its Heller twist. 

\begin{prop}  For a  stable Jordan type $\ul a = \sum_{i=1}^{p-1} a_i[i]$, 
denote by $\ul a^\perp$ the ``flip" of $\ul a$, 
$$\ul a^\perp =  \sum_{i=1}^{p-1} a_{p-i}[i].$$ 
Then 
$$ \ul{\JType}_{\Omega(M)}(v) = \ul{\JType}_M(v)^\perp, \quad v \in V(G).$$ 
\end{prop}

\begin{proof}
For any $v \in V(G)$, $\mu_v^*: (k(v)G-{\rm mod}) \to (k(v)\Gar-{\rm mod})$ is exact.  Moreover,
$\epsilon^*: (k\Gar-{\rm mod}) \to (k[u]/u^p-{\rm mod})$ is also exact.  Consequently, the existence
of a short exact sequence of the form \ $0 \to \Omega M \to P \to M \to 0$ with
$\JType_P(v) = N[p]$ for some $N$ implies the assertion.
\end{proof}

\begin{ex}  Let $g$ be a restricted Lie algebra with restricted enveloping algebra $u(g)$
(which is isomorphic to the group algebra of an infinitesimal group scheme of height 1).  Let
$\zeta$ be an even dimensional cohomology class in $\HHH^\bu(u(g),k)$, and $L_\zeta$ be the Carlson module defined by $\zeta$. 
Then $L_\zeta$ has two local Jordan types: 
it is generically projective (that is, the local Jordan type is $m[p]$  on a dense open set), and has the type $r[p] + [p-1]+ [1]$ on the hypersurface 
$\langle \zeta=0 \rangle$ in $\Spec \HHH^\bu(u(g),k)$.   Let $M$ be a $g$-module 
of constant Jordan type $\ul a$. 
Then the module $L_\zeta \otimes M$ has two local Jordan types: it is generically 
projective, and has the ``stably palindromic" type $\ul a + \ul a^\perp + [\proj]$ on 
$\langle \zeta = 0 \rangle$.
\end{ex}

We conclude this section with the following cautionary example
which shows why the construction of our local Jordan type function does not apply to
$kG$-modules $M$ for finite groups $G$.

\begin{ex}(\cite[2.3]{FPS})
\label{undefined}
Let $E = \bZ/p \times \bZ/p$, and write $kE = k[x,y]/(x^p,y^p)$.  Let $M = kE/(x-y^2)$.  Then 
$$\alpha: k[t]/t^p \to kE,\quad  t \mapsto x$$ and $$\alpha^\prime: k[t]/t^p \to kE, \quad t \mapsto x-y^2$$ 
are equivalent as $\pi$-points of $E$.  However, the Jordan type of $\alpha^*(M)$ equals 
$[\frac{p-1}{2}] + [\frac{p+1}{2}]$, whereas the Jordan type of $\alpha^{\prime *}(M)$ is $p[1]$. 
\end{ex}


\section{Maximal $j$--rank for arbitrary finite group schemes}
\label{j-type}

We begin with the following definition.

\begin{defn} 
Let $G$ be a  finite group scheme, $\alpha_K: K[t]/t^p \to KG$ be a  $\pi$-point of $G$,
and $j$ a positive integer with $1 \leq j < p$. 
Then $\alpha_K$ is said to be of maximal $j$-rank for some finite-dimensional $kG$-module
$M$ provided that the rank of $\alpha_K(t^j) = \alpha_K(t)^j: \  M_K \to M_K$ is greater or equal to 
the rank of $\beta_L(t^j): M_L \to M_L$ for any $\pi$-point $\beta_L: L[t]/t^p \to LG$.
\end{defn}

The purpose of this section is to establish in Theorem \ref{gen} that maximality of $j$-rank 
at $\alpha_K$ implies maximal $j$-rank at $\beta_L$ for any $\beta_L \sim \alpha_K$.
The proof consists of repeating almost verbatim the proof by A. Suslin and the authors
in \cite{FPS} of Theorem \ref{maximal}, so that we merely indicate here the explicit places
at which the proof of Theorem \ref{maximal} should be modified in order to prove Theorem
\ref{gen}.

The following theorem provides the key step.

\begin{thm} 
\label{max}
Let $k$ be an infinite field,  $M$ be a finite-dimensional $k$-vector space, and 
$\alpha, \alpha_1, \ldots, \alpha_n, \beta_1, \ldots, \beta_n$ be a family of  
commuting  nilpotent $k$-linear endomorphisms of $M$. Let $1\leq j \leq p-1$, and assume that  
$$\rk \alpha^j \geq \rk  (\alpha + \lambda_1\alpha_1 + \ldots + \lambda_n \alpha_n)^j$$ 
for any field extension $K/k$ and any $n$-tuple 
$(\lambda_1, \ldots,  \lambda_n) \in K^n$.  Then
$$\rk \alpha^j = \rk (\alpha + \alpha_1 \beta_1 + \ldots + \alpha_n \beta_n)^j.$$ 
In particular,  if $p(x, x_1, \ldots, x_n)$ is any polynomial  without constant  or linear term then 
$$\rk \alpha^j = \rk (\alpha + p(\alpha, \alpha_1, \ldots, \alpha_n))^j.$$
\end{thm}

\begin{proof} 
For $j=1$, this is \cite[1.8]{FPS}. For  general $j$, the statement 
follows by applying Corollary 1.11 of \cite{FPS}. 
\end{proof}

For any $\pi$-point $\alpha_K: K[t]/t^p \to KG$, we denote by $\rk(\alpha_K(t^j), M_K)$
the rank of the $K$-linear endomorphism $\alpha_K(t^j): M_K \to M_K$.

In the next 3 propositions, we consider the special cases in which $G$ is an elementary
abelian $p$-group, an abelian finite group scheme, and an infinitesimal finite group 
scheme.  In this manner, we follow the strategy of the proof of Theorem \ref{maximal}.

\begin{prop}
\label{elem-ab}
Let $E$ be an elementary abelian $p$-group of rank $r$,   let $M$ be a finite dimensional 
$kE$-module, 
and let $\alpha_K$ be a $\pi$-point of $E$ which is of maximal $j$-rank for $M$. 
Then for any $\beta_L \sim \alpha_K$, 
$$\rk(\alpha_K(t^j), M_K) = \rk(\beta_L(t^j), M_L).$$ 
\end{prop}

\begin{proof}
The proof of  \cite[2.7]{FPS} applies verbatim provided
one replaces references to  \cite[1.12]{FPS} by references to  \cite[1.9]{FPS}.
\end{proof}

\begin{prop}
\label{ab}
Let $C$ be an abelian finite group scheme over $k$, let $M$ be a finite dimensional $kC$-module, 
and let
$\alpha_K$ be a $\pi$-point of $C$  which is of maximal $j$-rank for $M$.  Then for any $\beta_L \sim \alpha_K$,
$$\rk(\alpha_K(t^j), M_K) = \rk(\beta_L(t^j), M_L).$$ 
\end{prop}

\begin{proof}
The proof of  \cite[2.9]{FPS} applies verbatim provided
one replaces references to  \cite[2.7]{FPS} by references to Proposition \ref{elem-ab}
and references to  \cite[1.12]{FPS} by references to  Theorem \ref{max}.
\end{proof}

\begin{prop}
\label{infin-max}
Let $G$ be an infinitesimal group scheme over $k$ and let $M$ be a finite 
dimensional $kG$-module.  Let $\beta_L: L[t]/t^p \to LG$ be 
a $\pi$-point of $G$ with the 
property that the $j$-rank of $\beta_L^*(M_L)$ is maximal for $M$. 
Then for any $\pi$-point $\alpha_K: K[t]/t^p \to KG$ which specializes 
to $\beta_L$,
$$\rk(\alpha_K(t^j), M_K) = \rk(\beta_L(t^j), M_L).$$ 
\end{prop}

\begin{proof}
The proof of  \cite[3.5]{FPS} applies verbatim provided
one replaces references to  \cite[2.9]{FPS} by references to Proposition~\ref{ab}.
\end{proof}

We now state and prove the assertion that maximality of $j$-rank at $\alpha_K$
implies maximality of $j$-rank at $\beta_L$ for any $\beta_L \sim \alpha_K$.   This
statement for all $j, 1 \leq j < p$, implies the maximality of Jordan
type as asserted in Theorem \ref{maximal}. 

\begin{thm}
\label{gen}
Let $G$ be a finite group scheme over $k$ and let $M$ be a finite 
dimensional $kG$-module.  Let $\alpha_K: K[t]/t^p \to KG$ be a $\pi$-point 
of $G$ which is of  
maximal $j$-rank for $M$. Then for any $\pi$-point 
$\beta_L: L[t]/t^p \to LG$ that specializes to $\alpha_K$, we have 
$$\rk(\alpha_K(t^j), M_K) = \rk(\beta_L(t^j), M_L).$$ 
\end{thm}

\begin{proof}
The proof of  \cite[4.10]{FPS} applies verbatim provided
one replaces references to  \cite[2.9]{FPS} by references to  Proposition~\ref{ab} and references to  \cite[3.5]{FPS} by references to  Proposition~\ref{infin-max}.
\end{proof}

We can now generalize the {\it modules of constant j-rank} as defined for infinitesimal group schemes in \cite{FP3} to all finite group schemes.

\begin{defn} 
\label{constant_rank}
A finite dimensional $kG$-module $M$ is said to be of {\it constant $j$-rank}, $1 \leq j < p$, if 
for any two $\pi$-points $\alpha_K: K[t]/t^p \to KG$, $\beta_L: L[t]/t^p \to LG$,  we have
$$\rk(\alpha_K(t^j), M_K) = \rk(\beta_L(t^j), M_L).$$
\end{defn} 

\begin{remark} By Theorem \ref{gen}, $M$ has constant $j$-rank $n$ if and only if
for each point of $\Pi(G)$ there is some $\pi$-point representative $\alpha_K$ with
$\rk(\alpha_K(t^j), M_K) = n$.
\end{remark} 

Evidently, a $kG$-module has constant Jordan type if and only if it has constant $j$-rank for 
all $j, 1 \leq j < p$ (see (\ref{ineq})).  

We shall say that $M$ is a {\it module of constant rank}  if it has constant $1$-rank. Every module of
constant Jordan type has, by definition, constant rank.  On the other hand, there are 
numerous examples of modules of constant rank which do not have constant Jordan type.
For example, if $\zeta \in \HHH^{2i+1}(G,k)$ is non-zero and $p > 2$, then the Carlson module
$L_\zeta$ is a $kG$-module of constant rank but not of constant Jordan type.

We finish this section with a cautionary example that illustrates that not all properties of maximal or constant Jordan type have natural analogues for 
maximal or constant rank.  Recall that a {\it generic Jordan type} of a $kG$-module $M$  is the Jordan type at a $\pi$-point  which represents a generic point of $\Pi(G)$.  By the main theorem of \cite{FPS}, it is well-defined. If $\Pi(G)$ is irreducible, we can therefore refer to the generic Jordan type of $M$. We can similarly define a {\it generic $j$-rank} of a $kG$-module to be $\rk(\alpha_K(t^j), M_K)$ for a $\pi$-point $\alpha$ of $G$ representing a generic point of $\Pi(G)$. By \cite[4.2]{FPS}, generic $j$-rank is well-defined.

\begin{ex}
\label{tensor}
Throughout this example we are using the formula for the tensor product of Jordan types (see, for example, \cite[Appendix]{CFP}).

\vspace{0.1in} (1). Let $\ul a = \sum a_i[i], \ul b = \sum b_i[i]$ be two Jordan types (or partitions) of the same cardinality.
In \cite[4.1]{CFP} the authors showed that $\ul a \geq \ul b$ implies $\ul a \otimes \ul c \geq \ul b \otimes  \ul c$ for any Jordan type $\ul c$. The  analogous statement is not true for ranks. 

Indeed, let $\ul a = 3[2]$, $\ul b =[3] + 3[1]$, and $\ul c = [2]$. Then 
$$\rk \ul a = 3 > \rk \ul b =2.$$
Since $\ul a \otimes \ul c = 3[3] + 3[1]$ and $\ul b \otimes \ul c  = [4] + 4[2]$, we have  
$$\rk \ul a \otimes \ul c = 6 < \rk \ul b \otimes \ul c = 7.$$

\vspace{0.1in} (2). The first part of this example illustrates a common failure of the upper semi-continuity property of the 
ranks of partitions with respect to tensor product.   Since this fails for partitions, it is reasonable to expect the 
same property  to fail for maximal ranks of modules.  The following is an explicit realization by $kG$-modules of this 
failure of upper semi-continuity. This example also shows that $M\otimes N$ can fail to have maximal rank at a 
$\pi$-point at which both $M$ and $N$ have maximal rank. This should be contrasted with the situation for maximal 
Jordan types (~\cite[4.2]{CFP}).

Let $G = \bG_{a(1)}^{\times 2} $ so that $kG \simeq k[x,y]/(x^p,y^p)$. 
Consider the $kG$-module $M$ of Example~\cite[2.4]{CFP}, pictured as follows: 
$$
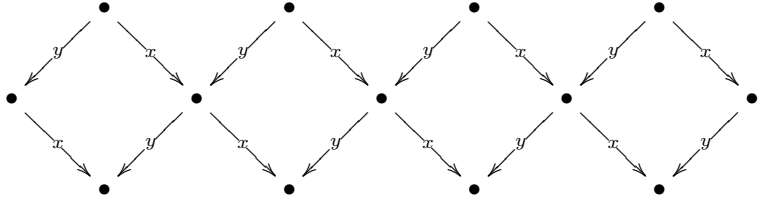
\begin{xy}*!C\xybox{%
\xymatrix{ &{\bu} \ar[dl]|y \ar[dr]|x 
&&{\bu} \ar[dl]|y \ar[dr]|x
&&{\bu} \ar[dl]|y\ar[dr]|x 
&&{\bu} \ar[dl]|y \ar[dr]|x\\
\bu \ar[dr]|x && \bu \ar[dl]|y \ar[dr]|x && \bu 
\ar[dl]|y \ar[dr]|x && \bu \ar[dl]|y \ar[dr]|x && \bu \ar[dl]|y \\
& \bu && \bu && \bu && \bu}}
\end{xy}
$$
Recall that $\Pi(G) \simeq \Proj \HHH^\bu(G,k)\simeq\mathbb P^1$. 
A point $[\lambda_1: \lambda_2]$ on $\mathbb P^1$ is represented by a $\pi$-point 
$\alpha: k[t]/t^p \to kG$ such that $\alpha(t) = \lambda_1 x + \lambda_2 y$. 

 For $p>5$, the module M has two Jordan types: the generic type $4[3] + 1[1]$ and the 
 singular type $3[3] + 2[2]$, which occurs at $[1:0]$ and $[0:1]$ (see \cite[2.4]{CFP}).  
 Hence, $M$ has constant rank.   We compute possible local Jordan types of $M \otimes M$ 
 using the fact that $\mu_{v*}: k(v)[t]/(t^p) \to k(v)G$ is a map
 of Hopf algebras for any $v \in V(G)$: 
\begin{enumerate}\item[(i)]
$(4[3] + 1[1])^{\otimes 2} = 16[5] + 24[3] + 17[1]$,
\item[(ii)]
$(3[3] + 2[2])^{\otimes 2} = 9[5] + 16[4] + 13[3] + 12[2] + 9[1].$
\end{enumerate} 
By \cite[4.4]{FPS}, the first type is the generic Jordan type of $M \otimes M$. Hence, the  
generic (and maximal) rank of $M \otimes M$ is 112. 
On the other hand, the rank of the second type is 110. Hence, the rank of $M$ at the 
points  $[1:0], [0:1]$ is maximal, but the rank of $M \otimes M$ is not.

\vspace{0.1in} (3).  Yet another result in \cite{CFP}, a direct consequence of the result 
on the tensor products of maximal types mentioned in (2), states that a tensor product 
of modules of constant Jordan type is a module of constant Jordan type.  This distinguishes 
the family of modules of constant Jordan type from the modules of constant rank, for which 
this property fails.  Let $M$ be the same as in (2). The calculation above shows that $M$ 
is of constant rank but $M \otimes M$ is not.

\vspace{0.1in}  We also give an example of a different nature, avoiding point by point calculations 
of Jordan types. This example was pointed out to us by the referee.  Let $M$ be a cyclic $kG$-module 
of dimension less than $p$ (e.g., $M = k[x,y]/(x^2, y^p)$). We have a short exact sequence $0 \to \Omega M \to kG \to M \to 0$.  
This implies that the local Jordan type of $\Omega M$ necessarily has $p$ blocks, and, hence, $\Omega M$ 
has constant rank. Since $\Omega M \otimes  \Omega^{-1}k \simeq M \oplus [\proj]$, we conclude that the 
tensor product of two modules of constant rank produces a module which is not of constant rank.

\end{ex}


\section{Refined support varieties for arbitrary finite group schemes}
\label{arbitrary}

In this section, we introduce the non-maximal
support varieties $\Gamma^j(G)_M$ for an arbitrary finite group scheme, finite
dimensional $kG$-module $M$, and integer $j, 1 \leq j < p$.  These are well
defined thanks to Theorem \ref{gen}.  After verifying a few simple properties of
these refined support varieties, we investigate various explicit examples.

\begin{defn} 
\label{gamma-j}  Let $G$ be a finite group scheme, and let
$M$ be a finite dimensional $kG$-module. Set
$$\Gamma^j(G)_M = \{[\alpha_K] \in \Pi(G) \, | \, 
\rk(\alpha_K(t^j), M_K) \text{ is not maximal} \},$$
the non-maximal $j$-rank variety of $M$.
\end{defn}

Our first example demonstrates that $\{ \Gamma^j(G)_M \}$
is a finer collection of geometric invariants than $\Pi(G)_M$.

\begin{ex}
\label{std-fp}
Let $G = \GL(3,\bF_p)$ with $p > 3$.   By \cite{Q} (see \cite[4.10]{FPS}),
the irreducible components
of $\Pi(G)$ are indexed by the conjugacy classes of maximal elementary 
$p$-subgroups of $G$ which are represented by subgroups of the unipotent 
group $U(3,\bF_p)$ of strictly upper triangular
matrices.  There are 3 such conjugacy classes, represented by the following subgroups:

\vspace{0.1in}
\noindent
$\left\{  
\left (
\begin{array}{cccccccc}
1&a&b\\
0&1&a\\
0&0&1\\
\end{array}
\right ) a,b \in \mathbb F_p \right \} \,
\left\{\left (
\begin{array}{cccccccc}
1&a&b\\
0&1&0\\
0&0&1\\
\end{array}
\right ) a,b \in \mathbb F_p \right \} \,
\left\{\left (
\begin{array}{cccccccc}
1&0&b\\
0&1&a\\
0&0&1\\
\end{array}
\right ) a,b \in \mathbb F_p \right \}   
$
\vspace{0.1in}

\noindent
Let $M$ be the second symmetric power of the standard 3-dimensional
 (rational) representation of $G$.  
Then the generic Jordan type of $M$ indexed by the first of these maximal 
elementary abelian subgroups 
of $G$  is $[3] + 3[1]$, whereas the Jordan types 
indexed by each of the other conjugacy classes of maximal elementary 
abelian $p$--subgroups
are $[2] + 4[1]$.

Thus, $\Pi(G)_M = \Pi(G)$ provides no information about $M$.

On the other hand, $\Gamma(G)_M = \Gamma^1(G)_M = \Gamma^2(G)_M$ equals the 
union of the two irreducible components of $\Pi(G)$ corresponding to the second and third
maximal elementary abelian $p$--subgroups, whereas $\Gamma^i(G)_M = \emptyset$ for $i > 2$.
\end{ex}
\vskip .1in

Our second example shows that $\Gamma^i(G)_M$ and $\Gamma^j(M)$ can be 
different, proper subsets of $\Pi(G)$.

\begin{ex}
In \cite[4.13]{FPS} A. Suslin  and the authors constructed an example of a finite group $G$ and a finite dimensional $G$-module $M$, such that $\Pi(G) = X \cup Y$   has two irreducible components and the generic Jordan types of $M$ at the generic points of $X$ and $Y$ respectively are incomparable.    Let $G$ and $M$ satisfy this property, and let $\alpha_K$ and $\beta_L$  be generic $\pi$-points of $X$ and $Y$ respectively.  
If $\JType(\alpha^*_K(M_K))$ and $\JType(\beta^*_L(M_L))$ are incomparable, then Lemma~\ref{reform} implies that there exist $i\not = j$ such that $\rk(\alpha_K(t^i), M_K) > \rk(\beta_L(t^i), M_L)$ but $\rk(\alpha_K(t^{j}), M_K) < \rk(\beta_L(t^{j}), M_L)$. Hence, $\Gamma^i(G)_M$ is a proper subvariety that contains the irreducible component $Y$ whereas $\Gamma^{j}(G)_M$  is a proper subvariety that contains the irreducible component $X$.    
\end{ex}

\vskip .1in

Our third example is a simple computation for a general finite group scheme. It provides another possible  ``pattern" for the varieties $\Gamma^i(G)_M$.

\begin{ex} Let $\zeta_1 \in \HHH^{n_1}(G,k)$ be an even dimensional class, and  $\zeta_2 \in \HHH^{n_2}(G,k)$ be an odd dimensional class.  Consider $L_{\ul \zeta} = L_{\zeta_1, \zeta_2}$, the kernel of the map
$$\zeta_1 + \zeta_2: \Omega^{n_1} k \oplus \Omega^{n_2} k \to k$$
The local Jordan type of $L_{\ul \zeta}$ at a $\pi$-point $\alpha$ is given in the following table: 

\noindent
$\left\{\begin{tabular}{ll}
$r[p]+[p-1]$, \quad & $\alpha^*(\zeta_1) \not = 0$\\
$r[p]+[p-2] + [1]$, \quad & $\alpha^*(\zeta_1)=0$, $\alpha^*(\zeta_2) \not = 0$\\
$(r-1)[p]+2[p-1] + [1]$, \quad & $\alpha^*(\zeta_1)=\alpha^*(\zeta_2)  = 0$\\
\end{tabular}
\right.$

\noindent
Hence, $\Gamma^1(G)_{L_{\ul \zeta}} = \ldots = 
\Gamma^{p-2}(G)_{L_{\ul \zeta}} = Z(\zeta_1)$, whereas 
$\Gamma^{p-1}(G)_{L_{\ul \zeta}} = Z(\zeta_1) \cap Z(\zeta_2)$, where $Z(\zeta_1)$ denotes the zero locus of a class $\zeta_1 \in \HHH^\bu(G,k)$ and $Z(\zeta_2)$ for $\zeta_2 \in \HHH^{\rm odd}(G,k)$ is defined in (\ref{Zzeta}).

\end{ex}

\vskip .1in
We next verify a few elementary properties of $M \mapsto \Gamma^j(G)_M$. Some of them are analogous to the properties of $\Gamma(G)_M$ stated in Prop~\ref{gamma}.

\begin{prop}
\label{closed}
Let $G$ be a finite group scheme and $M$ a finite dimensional
$kG$-module. 
\begin{enumerate}
\item\label{cl_pr} 
 $\Gamma^j(G)_M$ is a proper closed subset of $\Pi(G)$ for  $1 \leq j < p$.
 \item\label{cl_em}
$\Gamma^j(G)_M = \emptyset$ if and only if $M$ has constant $j$-rank. 
\item\label{cl_st} If $M$ and $N$ are stably isomorphic, then $\Gamma^j(G)_M = \Gamma^j(G)_N$ 
\item\label{cl_pl} If $M$ is a module of constant $j$-rank, then $\Gamma^j(G)_{M\oplus N} = \Gamma^j(G)_N$.
\item\label{cl_om} $\Gamma^j(G)_M = \Gamma^j(G)_{\Omega^2(M)}$. 
\item \label{cl_ga}
$\Gamma(G)_M \ = \ \cup_{1\leq j < p} \Gamma^j(G)_M.$
\item\label{cl_ge} If $M$  has the  Jordan type $m[p]$ at some generic $\pi$-point, then $\Gamma^1(G)_M = \ldots = \Gamma^{p-1}(G)_M = \Pi(G)_M$.\\
\end{enumerate}
\end{prop}

\begin{proof} 
By definition, $\Gamma^j(G)_M \subset \Pi(G)$ can never equal $\Pi(G)$, so it is a proper subvariety.
Moreover, assertions (\ref{cl_em}) and (\ref{cl_ga}) also immediately follow from definitions and 
Lemma~\ref{reform}. 
Assertion (\ref{cl_pl})  follows from the additivity of ranks and  of the functor 
$\alpha_K^*: KG-{\rm mod} \to K[t]/t^p-{\rm mod}$ induced by a $\pi$-point $\alpha_K$.  
Property (\ref{cl_st})  is proved exactly  as in the proof of Proposition~\ref{gamma}(1).

For (\ref{cl_om}), observe that a $\pi$-point $\alpha_K$ induces an exact functor and hence commutes with the Heller operator $\Omega$. The statement now follows from the observation that for $K[t]/t^p$-modules, applying $\Omega^{2}$ does not change the stable Jordan type.

To prove that $\Gamma^j(G)_M \subset \Pi(G)$ is closed as asserted in (\ref{cl_pr}),
we repeat the proof of \cite[5.2]{FPS} establishing that $\Gamma(G)_M$ is closed.
Indeed, the reduction in that proof to the special case in which $G$ is infinitesimal
applies without change.  The proof in the special case of $G$ infinitesimal uses
 the affine scheme of 1-parameter subgroups; this proof applies with only one
minor change: the set of equations on the ranks of powers of $f_A: A[t]/t^p \to \End_A(M)$
(in the notation of that proof) is replaced by the set of equations on rank of
only one, the $j$-th, power of $f_A$.

If $M$ is generically projective as in (\ref{cl_ge}), then $\Gamma(G)_M = \Pi(G)_M$. 
Let $\alpha_K \not \in \Gamma(G)_M$ so that the Jordan type of $\alpha_K^*(M)$ is $m[p]$, and  
let $\beta_L \in \Gamma(G)_M$. Let $\sum b_i[i]$ be the Jordan type  of $\beta_L^*(M_L)$.
 The statement follows easily from the formula (\ref{rank}):  we have 
$$\rk(\alpha_K(t^j), M_K) = m(p-j) >  \sum\limits_{i=j+1}^p b_i(i-j)=\rk(\beta_L(t^j), M_L),$$ where 
the inequality in the middle follows by downward induction on $j$ from the assumption 
$mp = \dim M =  \sum\limits_{i=1}^p b_ii.$
 Thus, $\Gamma^j(G)_M = \Gamma(G)_M$ for each $j, 1 \leq j < p$.

\end{proof}

\begin{ex} We point out that the ``natural" analog of \ref{gamma}(5) is not true for 
modules of constant rank. 
Namely, $\Gamma^1(G)_{M\otimes N}$ does not have to be equal to $\Gamma^1(G)_N$ 
for $M$ of constant rank. 
Indeed, let $M$ be as in Example~\ref{tensor}.  Then $M$ has constant rank and 
$\Gamma^1(E)_M = \emptyset$. But $\Gamma^1(E)_{M\otimes M} \not = \emptyset$ since 
$M \otimes M$ is not a  module of constant rank.  
\end{ex}

Using a recent result of R. Farnsteiner \cite[3.3.2]{Fa09}, we verify below that the non-maximal
subvarieties $\Gamma^i(G)_M \subset \Pi(G)$ of an indecomposable $kG$-module $M$
do not change when we replace $M$ by any
$N$ in the same component  as $M$ of the stable Auslander-Reiten quiver of $G$.  
This is a refinement of a result of the J. Carlson and the authors \cite[8.7]{CFP} which asserts that 
if $M$ is an indecomposable module of constant Jordan type than any $N$ in the same component  
of the stable Auslander-Reiten quiver of $G$ as $M$ is also of constant Jordan type.

\begin{prop}
\label{AR} Let $k$ be an algebraically closed field, 
 and $G$ be a finite group scheme over $k$. Let $\Theta \subset \Gamma_s(G)$ be a component of the stable Auslander-Reiten quiver of $G$. For any two modules $M, N$ in $\Theta$, and any $j, 1 \leq j \leq p-1$, $$\Gamma^j(G)_M = \Gamma^j(G)_N$$
\end{prop}

\begin{proof} 
Recall that $\Pi(G)$ is connected.  If $\dim \Pi(G) = 0$, then $\Pi(G)$
is a single point so that $\Gamma^j(G)_M$ is empty for any $kG$-module $M$.

Now, assume that $\Pi(G)$ is positive dimensional.  Since $k$ is assumed to be algebraically closed, to show that $\Gamma^j(G)_M = \Gamma^j(G)_N$, it's enough to show that their $k$-valued points are the same. For this reason, we shall only consider $\pi$-points defined over $k$. 

Let $M$ be a $kG$-module
in the component $\Theta$, and write the Jordan type of $\alpha^*(M)$ as
$\sum_{i=1}^p \alpha_i(M)[i]$.  
By \cite[3.1.1]{Fa09}, each component $\Theta$ determines non-negative
integer valued functions $d_i$ on the set of $\pi$-points (possibly different 
on equivalent $\pi$-points) and a positive, integer valued function $f$ on 
the modules occurring in $\Theta$ such that 
\begin{equation}
\label{d}
\left\{\begin{tabular}{l}
$\alpha_i(M) = d_i(\alpha) f(M)$ for $1\leq i \leq p-1$\\[2pt]
$\alpha_p(M) = \frac{1}{p}(\dim M - d_p(\alpha)f(M))$ \end{tabular}\right.
\end{equation}

Assume $[\beta]  \in \Gamma^j(G)_M$, so that
there exists a $\pi$-point $\alpha: k[t]/t^p \to kG$ such that $\rk\{\alpha^j(t), M \} > \rk\{\beta^j(t), M\}$. 
By (\ref{rank}),  this is equivalent to 
$$\sum\limits_{j=i+1}^p \alpha_i(M)(i-j) > \sum\limits_{j=i+1}^p \beta_i(M)(i-j).$$
Using formula (\ref{d}), we rewrite this inequality as 
$$\sum\limits_{j=i+1}^{p-1} d_i(\alpha)f(M)(i-j) + \frac{1}{p}(\dim M - d_p(\alpha)f(M))(p-j)  > $$
$$\sum\limits_{j=i+1}^{p-1} d_i(\beta)f(M)(i-j) + \frac{1}{p}(\dim M - d_p(\beta)f(M))(p-j).$$
Simplifying, we obtain
\begin{equation}
\label{e}
(\sum\limits_{j=i+1}^{p-1} d_i(\alpha)(i-j) - \frac{p-j}{p}d_p(\alpha))f(M) > (\sum\limits_{j=i+1}^{p-1} d_i(\beta)(i-j) - \frac{p-j}{p}d_p(\beta))f(M).
\end{equation}

Now, let $N$ be any other indecomposable $kG$-module in the component $\Theta$.
Multiplying the inequality (\ref{e}) by the positive, rational function $f(N)/f(M)$, 
we obtain the same inequality as (\ref{e})
 with $M$ replaced by $N$.  Thus, $[\beta] \in \Gamma^j(G)_N$.  Interchanging
the roles of $M$ and $N$, we conclude that $\Gamma^j(G)_M = \Gamma^j(G)_N$.
\end{proof}

For an infinitesimal group scheme $G$, the closed subvarieties $\Gamma^j(G)_M 
\subset \Pi(G)$ admit an affine version $V^j(G) \subset V(G)$ defined as follows

\begin{defn}
\label{aff-j-type}
Let $G$ be an infinitesimal group scheme, $M$ a finite dimensional $kG$-module,
and $j$ a positive integer, $1 \leq j < p$.  We define
$$V^j(G)_M = \{ v \in V(G)| \ \rk(\theta_v^j, M_{k(v)})
\text{ \ is \ not \ maximal}\} \cup \{0\} \subset V(G).
$$
(see \S\ref{refined} for notations). 
So defined, $V^j(G)_M - \{ 0 \}$ equals $\pr^{-1}(\Gamma^j(G)_M)$, where 
$\pr: V(G) - \{ 0 \} \to \Pi(G)$ is the natural (closed) projection (see \cite{FP2}).
\end{defn}

\begin{remark}
We can express $V^j(G)_M$ in terms of the locally closed subvarieties $V^{\ul a}(G)_M$
introduced in \S\ref{refined}.  Namely,  $V^j(G)_M$ is the union of  
$V^{\ul a}(G)_M \subset V(G)$ indexed by the Jordan types $\ul a$ with 
$\sum_{i=1}^p a_i\cdot i = \dim(M)$ satisfying the condition that there exists some Jordan
type $\ul b$ with $V^{\ul b}(G)_M \not= \{ 0 \}$ and $\sum_{i > j}^p b_i (i-j) > \sum_{i > j}^p a_i(i-j)$.
\end{remark} 

Our first representative example of $V^j(G)_M$  is a continuation of  (\ref{elemex}).

\begin{ex}
Let $G = \GL_{N(1)}$, let $M$ be the standard representation of  $\GL_N$, and assume $p$ does not divide $N$.   Recall that $V(\GL_{N(1)}) \simeq \cN_p$, where $\cN_p$ is the $p$-restricted nullcone of the Lie algebra $\gl_N$ (\cite[\S6]{SFB2}).  
The maximal Jordan type of $M$ is $r[p] + [N-rp]$, where $rp$ is the 
greatest non-negative multiple of $p$ which is less or equal to $ N$
(see \cite[4.15]{FPS}).  The rank of
the $j^{th}$-power of this matrix equals $r(p-j) + (N-rp-j) $ if $N-rp > j$
and $r(p-j) $ otherwise.   

For simplicity, assume $k$ is
algebraically closed so that we need only consider $k$-rational points
of $\cN_p$.   
For any $X \in \cN_p$, $\theta_X: M \to M$ is simply the
endomorphism $X$ itself.  Consequently, if $N-rp \leq j$, \
$V^j(G)_M \subset \cN_p$ consists of 0 together with
those non-zero  $p$-nilpotent $N\times N$ matrices
with the property that their Jordan types have strictly fewer than $r$ 
blocks of size $p$;\ if $N-rp > j$, then $V^j(G)_M $
consists of 0 together with  $0 \not= X \in \cN_p$
whose Jordan type is strictly less than $r[p] + [N-rp]$.   

Hence, the pattern for varieties $V^j(M)$ in this case  looks like
$$\emptyset \not =  V^1(G)_M = \ldots = V^{n}(G)_M 
\subset V^{n+1}(G)_M = \ldots = V^{p-1}(G)_M \subset V(G)$$
where $n = N-rp$.  
\end{ex}
\vskip .1in

Computing examples of $V^j(G)_M $ is made easier by
the presence of other structure.  For example, if $G = \cG_{(r)}$, the
$r^{th}$-Frobenius kernel of the algebraic group $\cG$ and if
the $kG$-module $M$ is the restriction of a rational $\cG$-module,
then we verify in the following proposition that $V^j(G)_M$ 
is $\cG$-stable, and thus a union of $\cG$-orbits inside $V(G)$.

\begin{lemma}
\label{stable}
Let $\cG$ be an algebraic group, and let $G$ be the $r^{th}$ 
Frobenius kernel of $\cG$ for some $r \geq 1$.  If $M$ is a finite dimensional rational
$\cG$-module, then each $V^j(G)_M, \ 1 \leq j < p,$
is a $\cG$-stable closed subvariety of $V(G)$.
\end{lemma}

\begin{proof}
Composition with the adjoint action of $\cG$ on $G$ determines an
action
$$\cG \times V(G) \ \to \ V(G).$$
 Observe
that for any field extension $K/k$ and any $x \in \cG(K)$, the pull-back
of $M_K$ via the conjugation action $\gamma_x: G_K \to G_K$ is 
isomorphic to $M_K$ as a $KG$-module.  Thus, the Jordan type of
$(\mu \circ \epsilon)^*(M_K)$ equals that of 
$(\gamma_x \circ \mu \circ \epsilon)^*(M_K)$ for any 1-parameter
subgroup $\mu: \bG_{a(r),K} \to G_K$.  
\end{proof}

Using Lemma \ref{stable}, we carry out
our second computation of $V^j(G)_M$ with $G$ infinitesimal,
this time for $G$ of height 2. 

\begin{ex}\label{nonmax}
Let $G =\SL_{2(2)}$. For simplicity, assume $k$ is algebraically closed. Recall that $$V(G) = \{ (\alpha_0,
\alpha_1) \, | \, \alpha_1, \alpha_2 \in sl_2, \alpha_1^p = \alpha_2^p =[\alpha_1, \alpha_2]=0 \},$$ the variety of pairs of commuting  $p$-nilpotent matrices (\cite{SFB1}).  The algebraic group $\SL_2$ acts on $V(G)$ by conjugation (on each entry). 
 
Let $e =
\left [\begin{array}{cc}
0 & 1 \\
0 & 0
\end{array} \right ]$. An easy calculation shows that the 
non-trivial orbits of $V(G)$
with respect to the conjugation
action are parameterized by $\mathbb P^1$, where $[s_0 : s_1] \in \mathbb
P^1$ corresponds to the orbit represented by the pair
$(s_0 e, s_1 e)$.

Let $S_\lambda$ be a simple  $\SL_2$-module of highest weight $\lambda$,  
$ 0 \leq  \lambda \leq p^2-1$.
Since $S_\lambda$ is a rational  $\SL_2$-module, the non-maximal rank varieties
$V^j(G)_{S_\lambda}$ are $\SL_2$-stable by Proposition \ref{stable}.   Hence, to 
compute the non-maximal rank varieties for $S_\lambda$
it suffices to   compute the Jordan type of $S_\lambda$ at the  orbit
representatives $(s_0 e, s_1 e)$.
By the explicit formula (\cite[2.6.5]{FP3}), the  Jordan type of  $S_\lambda$ at
$(s_0 e, s_1 e)$
is given by the Jordan type  of the nilpotent  operator $s_1 e  + s_0^p
e^{(p)}$ (here, $e^{(p)}$ is the divided power generator of $k\SL_{2(2)}$ as 
described in \cite[1.4]{FP3}).

The non-maximal rank varieties $V^j(G)_{S_\lambda}$ depend upon which of the following 
three conditions $\lambda$ satisfies. 
\noindent
\begin{enumerate}
  \item
\fbox{$0 \leq \lambda \leq p-1$}.   In this case, the Jordan type of  $e \in k\SL_{2(2)}$ 
as  an operator on $S_\lambda$
 is  $[\lambda + 1]$.
 On the other hand, the action of $e^{(p)}$ is
trivial.  Hence,  if $j \geq \lambda+1$, then the action $(s_1 e  + s_0^p
e^{(p)})^j$ is trivial for any pair $(s_0, s_1)$.   For $1 \leq j \leq \lambda$, the $j$-rank is maximal 
(and equals $\lambda +1-j$) whenever $s_1 \not =
0$. We conclude that for $j > \lambda$, we have $V^j(G)_{S_\lambda} = 0$, and for
$1 \leq j \leq \lambda$, $V^j(G)_{S_\lambda}$ is the orbit of $V(G)$ parametrized by
$[1:0]$.

\item \fbox{$p\leq \lambda < p^2 -1$}. Let $\lambda = \lambda_0 +
p\lambda_1$. By the Steinberg tensor product theorem, we have
$S_\lambda = S_{\lambda_0} \otimes S_{\lambda_1}^{(1)}$.  Observe that $e$
acts trivially  on $S_{\lambda_1}^{(1)}$ and
$e^{(p)}$ acts trivially on $S_{\lambda_0}$.  Moreover,
the  Jordan type of $e^{(p)}$
as an operator  on $S_{\lambda_1}^{(1)}$ is the same as the Jordan type  of
$e$ as an operator on $S_{\lambda_1}$.
Hence, the Jordan type of $s_1 e  + s_0^p e^{(p)}$  as  an operator on
$S_{\lambda_0} \otimes S_{\lambda_1}^{(1)}$  is
$[\lambda_0 +1] \otimes [\lambda_1 + 1]$ when
$s_0s_1\not=0$.
If $s_0=0$ or $s_1=0$ we get
the types $[\lambda_0 +1] \otimes (\triv)$ or $(\triv) \otimes [\lambda_1
+1]$ respectively.

\begin{enumerate} 
  \item For $0 < \lambda_0, \lambda_1 < p-1$, the tensor product  formula for Jordan types (see \cite[Appendix]{CFP}) 
implies that   the $j$-rank of $[\lambda_0 +1]
\otimes [\lambda_1 + 1]$
is strictly greater than that of $[\lambda_0 +1] \otimes (\triv)$ or
$(\triv) \otimes [\lambda_1 +1]$  for $j \leq \lambda_1 + \lambda_0 $. 
Hence, the non-maximal $j$-rank variety  in the   case when $j \leq \lambda_1 + \lambda_0 $ is a union of two orbits, 
parameterized by $[1:0]$ and $[0:1]$.
If $j > \lambda_1 + \lambda_0$, then the non-maximal $j$-rank variety is trivial 
since the $j$-rank  is $0$ at every point.

\item If $\lambda_0 =0$, then 
$S_\lambda \simeq S_{\lambda_1}^{(1)}$.  Hence, the  computation for 
$S_\lambda$ for $\lambda<p$
implies that the non-maximal $j$-rank variety in this case is  the orbit corresponding
to  $[0:1]$ for $j \leq \lambda_1$ and  is trivial otherwise.

\item For $\lambda_0 = p-1$ or $\lambda_1 = p-1$, the non-maximal j-rank variety is
the same as the support variety for any $j$, since the support variety is a proper
subvariety of $V(G)$ in this case. The support varieties for these modules were computed in \cite[\S 7]{SFB2} (see also \cite[1.17(4)]{FP3}).
\end{enumerate}

\item \fbox{$\lambda=p^2-1$}. In this case, $S_\lambda$  is the Steinberg module for
$\SL_{2(2)}$. Hence, it is projective, so
the non-maximal rank varieties are all  trivial.
\end{enumerate}

We summarize our calculations in the table below.  Let $\lambda = \lambda_0 + p\lambda_1$, and $\overline \lambda = \lambda_0 + \lambda_1$. 
If $j> \overline \lambda$, then $V^j(G)_{S_\lambda} = \emptyset$. For $j \leq \overline \lambda$, we have
$$
V^j(G)_{S_\lambda} =
\begin{cases}
\{ (\alpha_0, 0 ) \} \cup \{ (0, \alpha_1 )\} &  \text{if } 0< \lambda_0,
\lambda_1 < p-1 \\\{ (\alpha_0, 0 ) \} & \text{if $\lambda_0 \not = 0$, $\lambda_1 = 0$   or }
\lambda_0 = p-1,  \lambda_1 \not = p-1\\
\{ (0, \alpha_1 )  \} & \text{if $\lambda_0 = 0$, $\lambda_1 \not = 0$  or }
\lambda_0 \not = p-1, \lambda_1 = p-1\\
0 & \text{if $\lambda_0=\lambda_1=p-1$.}\\
\end{cases}
$$
In particular, for a given $\lambda= \lambda_0 + p\lambda_1$ we get the following pattern for 
$M = S_\lambda$: 
$$ V(G) \supset V^1(G)_M = \cdots   = V^{\bar \lambda}(G)_M 
\supset V^{\bar \lambda+1}(G)_M = \cdots = V^{p-1}(G)_M = \{0\}.$$
Observe that the only  simple modules 
of constant rank are the trivial module and the Steinberg module. An interested reader 
may find it instructive to compare this calculation to the calculation of support varieties 
for $\SL_{2(2)}$ (\cite[1.17(4)]{FP3}, see also \cite[\S7]{SFB2}).
\end{ex}


\section{Subvarieties of $\Pi(G)$ associated to individual $\Ext$-classes}
\label{ext-class}

For $M$ a $kG$-module of constant rank, we associate to a cohomology class $\zeta$ in  $\HHH^1(G,M)$ a closed
subvariety $Z(\zeta) \subset \Pi(G)$ which generalizes the construction of the
zero locus $Z(\zeta) \subset \Spec \HHH^\bu(G,k)$ of a homogeneous cohomology
class.  We show that this construction is intrinsically connected to the non-maximal rank 
variety, and establish some ``realization" results for non-maximal varieties as an application.   
Unless otherwise indicated, throughout this section $G$ will denote an arbitrary finite 
group scheme over $k$.

\begin{lemma}
\label{Ezeta} 
Let $M$ be a finite dimensional $kG$-module, and let $\zeta$ be a 
cohomology class  in $\HHH^1(G,M)$. Consider the corresponding extension 
$$\tilde \zeta: 0 \to M \to E_\zeta \to k \to 0.$$  
For any $\pi$-point $\alpha_K: K[t]/t^p \to KG$, the following are equivalent:
\begin{itemize}
\item[(i)]
the cohomology class 
$\alpha_K^*(\zeta_K)\in \HHH^1(K[t]/t^p, M_K)$ is trivial. 
\item[(ii)]
$\rk(\alpha^*_K(t), E_\zeta)  \ = \  \rk (\alpha^*_K(t),M)$.
\item[(iii)]
$\JType(\alpha_K^*(E_{\zeta,K}))\  =  \ \JType(\alpha_K^*(M_K)) + 1[1]$. \end{itemize}
\end{lemma}

\begin{proof}
Recall that  $\alpha_K^*(-)$ is exact (by definition, $\alpha_K$ is flat); moreover, the
sequence $\alpha_K^*(\tilde \zeta)$ splits  if and only if $\alpha_K^*(\zeta) = 0$ in $ 
\HHH^1(K[t]/t^p,K)$.  Thus, it suffices to prove that a short exact sequence 
$0 \to M \to E \to K \to 0$ of $K[t]/t^p$-modules splits if and only if 
$\rk ( t, M ) = \rk ( t, E )$ if and only if $\JType(E) = \JType(M) + 1[1]$.
Let $\ul b = \sum_{i=1}^p b_i[i]$ be the Jordan type of $E$ and $\ul a = \sum_{i=1}^p a_i [i]$
be the Jordan type of $M$.   Then this short exact sequence splits
if and only if the map $E \to k$ factors through the summand $b_1[1]$ of $E$ which
occurs if and only if $b_i = a_i, i > 1$ which is equivalent to $\rk (t, M ) = \rk ( t, E )$.
\end{proof}

\begin{prop}
\label{welldef}
Let $M$ be a $kG$-module of constant rank,  and let $\zeta$ be a 
cohomology class  in $\HHH^1(G,M)$.  Consider the corresponding extension 
$$\tilde \zeta: 0 \to M \to E_\zeta \to k \to 0.$$ 
\begin{enumerate}
\item If $E_\zeta$ has constant rank equal to that of $M$, then 
$\alpha_K^*(\zeta_K)\in \HHH^1(K[t]/t^p, M)$ is trivial for every 
$\pi$-point $\alpha_K: K[t]/t^p \to KG$.
\item If $E_\zeta$ has constant rank greater than that of $M$, then 
$\alpha_K^*(\zeta_K)\in \HHH^1(K[t]/t^p, M)$ is non-trivial for every 
$\pi$-point $\alpha_K: K[t]/t^p \to KG$.
\item If $E_\zeta$ does not have constant rank, then $\alpha_K^*(\zeta)$ is
trivial if and only if $[\alpha_K] \in \Gamma^1(G)_{E_\zeta} \subset \Pi(G)$.
\item For any two equivalent $\pi$-points $\alpha_K$, $\beta_L$ of $G$, 
$\alpha_K^*(\zeta_K)$ is trivial if and only if $\beta_L^*(\zeta_L)$ is trivial.
\end{enumerate}

\end{prop}

\begin{proof}
Assertions (1) and (2) follow immediately from Lemma \ref{Ezeta}.  Assertion (3)
also follows from Lemma \ref{Ezeta}: if $E_\zeta$ does not have constant rank, then 
the complement of 
$\Gamma^1(G)_{E_\zeta}$ in $\Pi(G)$
consists of those equivalence classes of $\pi$-points $\alpha_K$
satisfying Lemma \ref{Ezeta}(ii.).  

To prove that the vanishing of $\alpha_K^*(\zeta_K)$ depends only upon the equivalence 
class of $\alpha_K$, we examine each of the three cases considered above.  In case (1),
 $\alpha_K^*(\zeta_K) = 0$ for all $\pi$-points $\alpha_K$: on the other hand, in case (2)
$\alpha_K^*(\zeta_K) \not= 0$ for all $\pi$-points $\alpha_K$.  Finally, the assertion in case (3)
follows immediately from Theorem \ref{gen}.
\end{proof}

Proposition \ref{welldef}(4) justifies the following definition.

\begin{defn}
\label{Zzeta}
For $M$ a module of constant rank, and $\zeta \in \HHH^1(G,M)$, 
we define 
\begin{equation} 
\label{one}
Z(\zeta) \ \equiv \ \{[\alpha_K] \ | \ \alpha_K^*(\zeta) = 0 \} \ \subset \ \Pi(G).
\end{equation}

For $\zeta \in \HHH^m(G,k)$, we define
\begin{equation} 
\label{two} 
Z(\zeta) \ \equiv \ \{[\alpha_K] \ | \ \alpha_K^*(\zeta) = 0 \} \ \subset \ \Pi(G).
\end{equation}
\end{defn}
\vspace{0.2in}

Since $\HHH^m(G,k) \simeq \HHH^1(G, \Omega^{1-m}k)$, the definition of (\ref{two})
is a special case of that of (\ref{one}).
For $m = 2n$ even, $Z(\zeta)$ corresponds under the isomorphism $\Pi(G) \simeq
\Proj \HHH^\bu(G,k)$ with 
the hypersurface $\langle \zeta=0 \rangle$ in $\Spec \HHH^{\bu}(G,k)$
as shown below in Proposition~\ref{comparison}.  
\vspace{0.1in}

\begin{remark}
We point out that Definition \ref{Zzeta} is not as straight-forward as it might appear.
\begin{itemize}
\item
Let $G = \bZ/p \times \bZ/p$ with $p > 2$, write $kG = k[x,y]/(x^p,y^p)$ and consider 
$M = kG/(x-y^2)$ as in Example \ref{undefined}.    Consider the short exact sequence
$$0 \to \Rad(M) \to M \to k \to 0,$$
with associated extension class $\zeta \in \HHH^1(G,\Rad(M))$.  Consider the equivalent
$\pi$-points  $\alpha, \alpha^\prime: k[t]/t^p \to kG$ of Example \ref{undefined}.
Then, $\alpha^*(\zeta) \not= 0$, yet $\alpha^{\prime *}(\zeta) = 0$.  Thus, the ``zero locus" 
of $\zeta$ is not a well defined subset of $\Pi(G)$.
 \item
 Let $\zeta \in \HHH^{2n}(G,k)$ represented by $\hat \zeta: \Omega^{2n}k \to k$.   By definition
 of  $L_\zeta$, we have an extension 
 $$
 \tilde \xi: \ 0 \to L_\zeta \to \Omega^{2n}k \stackrel{\hat\zeta}{\to} k \to 0,
 $$
 corresponding to a cohomology class $\xi \in \HHH^1(G,L_\zeta)$.  Then for any $\pi$-point
 $\alpha_K: K[t]/t^p \to KG$, $\alpha_K^*(\tilde \xi)$ splits if and only if $\alpha_K^*(L_\zeta)$ is free if and only if $[\alpha_K] \not  \in \Pi(G)_{L_\zeta}$ if and only of $\alpha_K^*(\zeta) \not = 0$.
 Thus, the zero locus of $\xi$ equals the {\it complement} of the zero locus
 of $\zeta$ (and thus is open in $\Pi(G))$.
 \item
 For $\zeta \in \HHH^{2n+1}(G,k)$, one could define $Z(\zeta)$ as the zero locus of the Bockstein
 of $\zeta$ provided one is in a situation in which the Bockstein is defined and well
 behaved.  See the discussion of the Bockstein following Example \ref{int}.
 \end{itemize}
\end{remark}

We recall from \cite{CF} that a short exact sequence of $kG$ modules
$$\tilde \xi: \quad 0 \to M \to E \to Q \to 0$$ 
is said to be {\it locally split} if 
$\alpha_K^*(\tilde \xi)$ splits for every $\pi$-point $\alpha_K: K[t]/t^p \to KG$ of $G$.

\begin{prop} 
\label{Zclosed}
Let $M$ be a module of constant rank, and let $\zeta$ be a 
cohomology class  in $\HHH^1(G,M)$. Consider the corresponding extension 
$$\tilde \zeta: 0 \to M \to E_\zeta \to k \to 0.$$   
Then 
$$  Z(\zeta) = \quad
\begin{cases}
\Pi(G), \quad {\text if \ } \tilde \zeta {\text \ is  \ locally \ split}\\
\Gamma^1(G)_{E_\zeta}, \quad {\text if \ } \tilde \zeta {\text \ is \ not \ locally \ split}.\\
\end{cases}
$$

In particular, $Z(\zeta) \ \subset \ \Pi(G)$ is closed.
\end{prop}

\begin{proof}
Observe that $\tilde \zeta$ is split at $[\alpha_K]$ if and only if $\alpha_K^*(\zeta) = 0$.
We first consider $\zeta$ such that $E_\zeta$ has constant rank.  Then by Proposition
\ref{welldef}.1, $Z(\zeta)$ equals $\Pi(G)$ if $\tilde \zeta$ is locally split and $Z(\zeta) =
\emptyset$ by Proposition \ref{welldef}.2 if $\tilde \zeta$ is not locally split.  Alternatively,
if $E_\zeta$ does not have constant rank, then Proposition \ref{welldef}.3 gives the 
asserted description of $Z(\zeta)$.

Because $\Gamma^1(G)_{E_\zeta} \subset \Pi(G)$ is closed by Proposition \ref{closed}
and of course $\Pi(G)$ is itself closed in $\Pi(G)$, we conclude that $Z(\zeta)$ is closed
inside $\Pi(G)$.
\end{proof} 

We remark that $\zeta \in \HHH^1(G,M)$ can be non-zero and yet $Z(\zeta) = \emptyset$.  To say
$Z(\zeta) = \emptyset$ is to say that $\alpha_K^*(\zeta) = 0$  for all $\pi$-points $\alpha_K$.
Consider, for example,  an even dimensional non-trivial cohomology class 
$\zeta \in \HHH^{2n}(G,k)$ which is a product of odd dimensional classes. 
Since the product of any two odd classes in $\HHH^*(k[t]/t^p,k)$ is zero, 
$\alpha_K^*(\zeta) = 0$  for all $\pi$-points $\alpha_K$ of $G$. On the other hand, 
$\zeta$ can be identified with  a cohomology class in 
$\HHH^{1}(G, \Omega^{1-2n}(k)) \simeq \HHH^{2n}(G,k)$. Since $\Omega^{1-2n}(k)$ 
is a module of constant Jordan type (see \cite{CFP}), the class $\zeta$ satisfies the 
requirements of Proposition \ref{closed}.

A more interesting example is the following.

\begin{ex} 
\label{int}
Let $G$ be a finite group scheme with the property that the dimension of $\Pi(G)$ is
at least 1.  Let $\zeta^\prime \in \widehat \HHH^{-i}(G,k), \ i > 0$, be an element in the
negative Tate cohomology of $G$.  As shown in \cite[6.3]{CFP}, $\alpha_K^*(\zeta^\prime)
= 0$ for any $\pi$-point $\alpha_K$.  Then $\zeta^\prime$
corresponds to $\zeta \in \HHH^1(G,\Omega^{i+1}(k))$ under the isomorphism
$\HHH^{-i}(G,k) \simeq \HHH^1(G,\Omega^{i+1}(k))$; by the naturality of this isomorphism,
$\alpha_K^*(\zeta) = 0 \in \widehat \HHH^{-i}(K[t]/t^p,K)$ for any $\pi$-point $\alpha_K$. 

Thus, $\zeta \not = 0, \ \tilde \zeta$ is locally split, and $Z(\zeta) \ = \ \emptyset$ for this choice 
of $\zeta \in \HHH^1(G,\Omega^{i+1}(k))$.
\end{ex}

For any field extension $K/k$, let $R_K=
W_2(K)$ denote the Witt vectors of length 2 for $K$.
Assume that $G$ over $k$ embeds into an $R_k$-group scheme $G_{R_k}$ so 
that $G = G_{R_k} \times_{\Spec R_k} \Spec k \ \subset G_{R_k}$, thereby inducing by base change
$G_K \subset G_{R_K}$.   Then we may define the Bockstein $\beta: \HHH^i(G_K,K) \to
\HHH^{i+1}(G_K,K)$ for $i>0$ as the connecting homomorphism for the short exact sequence
of $G_{R_K}$-modules
\begin{equation}
\label{witt}
0 \to K \ \to \ R_K \ \to \ K \to 0.
\end{equation}
(The reader is referred to \cite[3.4]{Evens} for a discussion of this Bockstein.)
Since any $\pi$-point $\alpha_K: K[t]/t^p \to KG$ lifts to a map $\tilde \alpha_K: R_K[t]/t^p 
\to R_KG_{R_K}$ of $R$-algebras, $\alpha^*: \HHH^*(G,K) \to \HHH^*(K[t]/t^p,K)$ commutes 
with this Bockstein.
Since $\beta: \HHH^{2d-1}(K[t]/t^p,K) \to \HHH^{2d}(K[t]/t^p,K)$ is an isomorphism, we conclude that 
if $x \in \HHH^{2d-1}(G,k)$, then $\alpha_K^*(x)$ vanishes if and only if $\alpha_K^*(\beta(x)) = 0$, 
where $\beta(x) \in \HHH^{2d}(G,k)$.  
Thus, for such $G$ lifting to $G_{R_k}$ and for $p > 2$, when considering $Z(\zeta)$ for 
homogeneous classes in $\HHH^*(G,k)$, it suffices to
restrict attention to the subalgebra $\HHH^\bu(G,k)$ of even dimensional classes.

As we see in the following proposition, Definition \ref{Zzeta} of $Z(\zeta)$ 
extends the ``classical" definition of the vanishing locus of a (homogeneous) cohomology
class in $\HHH^\bu(G,k)$.

\begin{prop}
\label{comparison}
Let $n$ be a positive integer, and set $M = \Omega^{1-2n}(k)$. Let $\zeta \in \HHH^1(G,M)$,
and let $\zeta^\prime \in \HHH^{2n}(G,k)$ be the corresponding element under the natural isomorphism
$\HHH^1(G,M) \simeq \HHH^{2n}(G,k)$.  Then the isomorphism $\Pi(G) \simeq 
\Proj \HHH^\bu(G,k)$ of Theorem~\ref{cohom} restricts to an isomorphism
$$Z(\zeta) \ = \ \Proj \HHH^\bu(G,k)/(\zeta^\prime).$$
\end{prop}

\begin{proof}  Let $L_{\zeta^\prime}$ be the Carlson module associated to the class $\zeta^\prime$.
The exact triangle $$\tilde \zeta: \Omega^{1-2n}(k) \to E_\zeta \to k \to \Omega^{-2n}(k)$$ corresponds to the exact triangle 
$$\xymatrix{\tilde \zeta^\prime: \Omega^1(k) \ar[r]& L_{\zeta^\prime} \ar[r]& \Omega^{2n}(k) \ar[r]^-{\zeta^\prime}& k}$$ under the shift $\Omega^{2n}$. Hence, $L_{\zeta^\prime}$ is stably isomorphic to $\Omega^{2n}(E_\zeta)$.  
By Prop.~\ref{closed}, 
\begin{equation}\label{*}
\Gamma^1(G)_{E_\zeta} = \Gamma^1(G)_{L_{\zeta^\prime}}.
\end{equation} 

If $\tilde\zeta$ is locally split, then so is $\tilde\zeta^\prime$ by the naturality of the isomorphism $\HHH^1(G,M) \simeq \HHH^{2n}(G,k)$. This implies that $\zeta^\prime$  is nilpotent by the ``Nilpotence detection theorem" of Suslin (\cite{S}). Hence, in this case $\Proj \HHH^\bu(G,k)/(\zeta^\prime) =\Proj \HHH^\bu(G,k) \simeq \Pi(G)$. By Prop.~\ref{Zclosed}, $Z(\zeta) \simeq \Pi(G)$ as well. Hence, in this case $Z(\zeta) = \Pi(G) \simeq \Proj \HHH^\bu(G,k)/(\zeta^\prime)$. 

If $\tilde\zeta$ is not locally split, then $Z(\zeta) = \Gamma^1(G)_{E_\zeta}$ by Proposition~\ref{Zclosed}. Since $L_{\zeta^\prime}$ is generically projective, Proposition~\ref{closed} implies that $\Gamma^1(G)_{L_{\zeta^\prime}} = \Pi(G)_{L_{\zeta^\prime}}$.   By \cite[2.9]{FP2} (see \cite{C} for finite groups), $\Pi(G)_{L_\zeta^\prime} \simeq  \Proj \HHH^\bu(G,k)/(\zeta^\prime)$ under the isomorphism $\Phi_G$ of (\ref{cohom}). 
The equality (\ref{*}) now implies $Z(\zeta) \simeq \Proj \HHH^\bu(G,k)/(\zeta^\prime)$. 

\end{proof}


\begin{prop}
\label{Gammareal}
Let $G$ be a finite group scheme over $k$.
Let $\zeta_i \in \HHH^{2d_i+1}(G,k) \simeq \HHH^1(G,\Omega^{-2d_i}k), 
\ 1 \leq i \leq r, d_i\geq 0 $.  
Let $M = \oplus_{i=1}^r \Omega^{-2d_i}k$, and set $\zeta = \oplus_i \zeta_i
\in \HHH^1(G,M)= \oplus_i \HHH^1(G, \Omega^{-2d_i}k)$.  Let $$0 \to M \to E_{\zeta} \to k \to 0$$ 
be the corresponding extension.
Then 
$$\Gamma^1(G)_{E_\zeta}  \ = \ Z(\zeta) \ = \ Z(\zeta_1)\cap\ldots\cap Z(\zeta_r),$$
and $\Pi(G)_{E_\zeta} = \Pi(G)$.
\end{prop}

\begin{proof}
To prove (1), observe that  Lemma~\ref{Ezeta}(1)
implies that $\Gamma^1(G)_{E_\zeta} = \{[\alpha_K] \ | \ \alpha_K^*(\zeta) = 0 \}$. 
Since $\zeta = \oplus \zeta_i$, we further conclude  $\{[\alpha_K] \ | \ \alpha_K^*(\zeta) = 0 \} = 
\{[\alpha_K] \ | \ \alpha_K^*(\zeta_i) = 0 \text{ for all } i\} = \bigcap\limits_i Z(\zeta_i)$.
\sloppy{

}

To verify that $\Pi(G)_{E_\zeta} = \Pi(G)$, we observe that the the generic Jordan type of 
$E_{\zeta}$ is of the form $m[p] + [2] + (r-1)[1]$ at generic points $[\alpha_K] \in \Pi(G)$
such that $\alpha_K^*(\zeta) \not= 0$ and of the form $m[p] + (r+1)[1]$ otherwise.  
This follows immediately from the observation that $\Omega^{-2d_i}(k)$ has
constant Jordan type of the form $m_i[p] + [1]$, and thus $M$ has constant (and, in
particular, generic) Jordan type $(\sum_i m_i)[p] + r[1]$.  

\end{proof}

As we see below,  
the construction of $E_{\ul \zeta}$ in Proposition ~\ref{Gammareal} above is in fact a generalized Carlson 
module  $L_{\underline\zeta}$ (as defined in \cite{CFP}) ``in disguise".  This phenomenon has already 
appeared  in the proof of Proposition~\ref{comparison}
 for a single cohomology class $\zeta$. 
Since this construction applies to 
homogeneous cohomology classes
$\zeta_i$ which are either all in even degree or all in odd degree, and since 
Proposition \ref{Gammareal} discusses classes of odd degree, we consider in Example
\ref{even} classes $\zeta_i$ in even degree.

\begin{ex}
\label{even}
Let $\underline \zeta = (\zeta_1, \ldots, \zeta_r)$, where 
$\zeta_i \in \HHH^{2d_i}(G,k) \simeq \underline{\Hom}(\Omega^{2d_i}(k),k), 
\ 1 \leq i \leq r, d_i\geq 0$. Let $ L_{\underline\zeta}$ be the kernel of the  map 
$\zeta = \sum \zeta_i: \bigoplus \Omega^{2d_i}(k)  \to k$, so that we have an exact sequence:
$$\xymatrix{0\ar[r] &  L_{\underline\zeta} \ar[r] &  \bigoplus \Omega^{2d_i}(k) 
\ar[rr]^-{\zeta_1+\cdots +\zeta_r} && k \ar[r] & 0 }$$
This short exact sequence represents an exact triangle in $\stmod kG$. 
Shifting the triangle by $\Omega^{-1}$ we obtain a triangle 
$$\xymatrix{k \ar[r] & \Omega^{-1}(L_{\underline\zeta}) \ar[r] &  
\bigoplus \Omega^{2d_i-1}(k) \ar[r] & \Omega^{-1}(k) &  }$$
Hence, $\underline \zeta$ corresponds to a short exact sequence 
$$\xymatrix{0 \ar[r] & k \ar[r] & F_{\underline \zeta} \ar[r] &  
\bigoplus \Omega^{2d_i-1}(k) \ar[r]& 0 }$$
with the middle term stably isomorphic to $\Omega^{-1}(L_{\underline\zeta}) $.  
Taking the dual of this short exact sequence,  we 
obtain the analogue with even dimensional cohomology classes of the short 
exact sequence which defines $E_\zeta$ as in Proposition~\ref{Gammareal}    (but for odd classes): 
$$\xymatrix{0 \ar[r]&   \bigoplus \Omega^{1-2d_i}k  \ar[r] & E_\zeta \ar[r] & k \ar[r] &0}.$$
Hence, $E_\zeta$ is stably isomorphic to $\Omega^{-1} (L_{\underline\zeta}^\#)$. 
\end{ex} 

Our final result extends the construction of closed zero loci to 
extension classes $\xi \in \Ext_G^n(N,M)$ with both 
$M, \ N$ of constant Jordan type.   In other words, Proposition \ref{extension} introduces
the (closed) support variety $Z(\xi)$ of such an extension class.

\begin{prop}
\label{extension} Let $G$ be a finite group scheme and $N, M$ finite dimensional 
$kG$-modules of constant Jordan type. Let $\xi \in \Ext^n_G(N,M) \simeq 
\Ext^1(\Omega^{n-1}(N),M)$ for some $n \not= 0$, 
and consider the corresponding extension
$$\tilde \xi: \ \ 0 \to M \to E_\xi \to \Omega^{n-1}(N) \to 0.$$

\noindent
(1) If  $\alpha_K , \ \beta_L$ are equivalent $\pi$-points of $G$, then 
$\alpha_K^*(\tilde \xi)$ splits if and only if $\beta_L(\tilde \xi)$ splits. 

\noindent
(2) If $$Z(\xi) \ \equiv \ \{[\alpha_K] \ | \ \alpha_K^*(\tilde \xi) {\ \text splits} \} \ \subset \ \Pi(G),$$
then 
$$  Z(\xi) = \quad
\begin{cases}
\Pi(G), \quad {\text if \ } \tilde \xi {\text \ is  \ locally \ split}\\
\Gamma^1(G)_{E_\xi}, \quad {\text if \ } \tilde \xi {\text \ is \ not \ locally \ split}.\\
\end{cases}
$$
\end{prop}

\begin{proof}
There is a natural isomorphism
$$\Ext_G^1(\Omega^{n-1}(N),M) \ \simeq \ \HHH^1(G,(\Omega^{n-1}(N))^\#\otimes M)$$ 
sending the extension class $\xi$ to the cohomology class 
$\zeta \in \HHH^1(G,(\Omega^{n-1}(N))^\#\otimes M)$ (where $(\Omega^{n-1}(N))^\#$ is the linear dual of $\Omega^{n-1}(N)$).
Hence, $\alpha_K^*(\tilde \xi)$ splits if and only $\alpha_K^*(\tilde \zeta)$ splits for any
$\pi$-point $\alpha_K$ of $G$.

By \cite[5.2]{CFS}, $(\Omega^{n-1}(N))^\#$ has constant Jordan type.  Thus, by \cite[4.3]{CFS},
$(\Omega^{n-1}(N))^\#\otimes M$ also has constant Jordan type.  Consequently, the assertion of
the Proposition for $\xi$ follows from Proposition \ref{closed} for $\zeta$.
\sloppy{

}

\end{proof}


\end{document}